\documentclass[10pt]{amsart}

\usepackage{amsmath}
\usepackage{amssymb}
\usepackage{amsfonts}
\usepackage{geometry}
\usepackage{url}
\usepackage{amsmath}
\usepackage{young}

\usepackage{savesym}
\usepackage{xypic}
\savesymbol{none}
\usepackage{ytableau}

\usepackage{setspace}
\usepackage{amsthm}
\usepackage[all,cmtip]{xy}
\usepackage{amsmath,amscd}
\usepackage{thmtools}
\usepackage{thm-restate}
\usepackage[enableskew]{youngtab}
\usepackage{verbatim}
\usepackage{hyperref}

\newtheorem{theorem}{Theorem}[section]
\newtheorem{lemma}[theorem]{Lemma}
\newtheorem{proposition}[theorem]{Proposition}

\newtheorem{remark}[theorem]{Remark}
\newtheorem*{nono-theorem}{Theorem}

\newtheorem{definition}[theorem]{Definition}

\theoremstyle{definition}
\newtheorem{example}[theorem]{Example}

\numberwithin{equation}{section}

\newcommand{\Ker}{\mathrm{Ker}}

\newcommand{\CC}{\mathbb{C}}
\newcommand{\ZZ}{\mathbb{Z}}
\newcommand{\NN}{\mathbb{N}}
\renewcommand{\H}{\ddot{H}_n\left(\kappa\right)}
\newcommand{\Ind}{\mathrm{Ind}}
\newcommand{\W}{\dot{W}_n}

\renewcommand{\u}{u_1,\ldots ,u_n}
\newcommand{\z}{z_1,\ldots ,z_n}
\newcommand{\One}{\mathbf{1}}
\newcommand{\w}{\tilde{w}}

\begin{document}

\title{Irreducible modules for the degenerate double affine Hecke algebra of type $A$ as submodules of Verma modules}

\author{Martina Balagovi\' c}
\address{\newline
School of Mathematics and Statistics\\  Newcastle University\\ UK }
\thanks{martina.balagovic@newcastle.ac.uk} 
\email{martina.balagovic@newcastle.ac.uk}
%\subjclass[2010]{Primary 17B10, 17B37, 81R50, 20G42}

\begin{abstract}
We give a full classification, in terms of periodic skew diagrams, of irreducible modules in category $\mathcal{O}_{ss}$ for the degenerate double affine Hecke algebra of type $A$ which can be realised as submodules of Verma modules. 
\end{abstract}

\maketitle

\section{ Introduction} 
Irreducible representations in category $\mathcal{O}$ for the degenerate double affine Hecke algebra (trigonometric Cherednik algebra) $\H$ of type $A$  have been classified by Suzuki in \cite{S2}. They are parametrized by \emph{periodic Cherednik diagrams}, which are an infinite, periodic, skew generalization of Young diagrams. Given such a diagram $D$, the construction in \cite{S2} produces a character $\chi_D$ of a commutative subalgebra $\CC[u_1,\ldots, u_n]$ of $\H$. One can induce from any one dimensional representation $\chi$ of the subalgebra $\CC[u_1,\ldots, u_n]$ and get a representation of the entire algebra $\H$; the resulting induced representation $M_{\chi}$ is called the \emph{Verma module}. For a character $\chi=\chi_D$ obtained from a periodic Cherednik diagram $D$, the resulting Verma module $M_\chi=M_D$ has a distinguished quotient $N_D$ called the \emph{small Verma module}, which in turn has a unique irreducible quotient $L_{D}$. All irreducible representations of $\H$ in the appropriately defined category $\mathcal{O}$ can be realised in this way, and it is known which diagrams $D$ produce isomorphic irreducible representations. This gives a parametrization of  simple objects in category $\mathcal{O}$ in terms of equivalence classes of periodic Cherednik diagrams. 

A certain full subcategory $\mathcal{O}_{ss}$ of $\mathcal{O}$ is particularly approachable. Its objects are \emph{semisimple} or \emph{calibrated} modules, defined as those $\H$ modules on which the subalgebra $\CC[u_1,\ldots, u_n]$ acts diagonally. Suzuki and Vazirani \cite{SV} classify such irreducible modules in terms of periodic Cherednik diagrams. They prove that an irreducible module $L_D$ is semisimple if and only if $D$ is a \emph{periodic skew diagram}. This is a combinatorial condition on the arrangement of boxes in $D$, and it directly generalizes the corresponding condition for semisimple representations of the degenerate affine Hecke algebra from \cite{R}.

In fact, many results in the representation theory of (degenerate) double affine Hecke algebras parallel analogous results in the representation theory of (degenerate) affine Hecke algebras. The correspondence is analogous to that between Weyl groups and an affine Weyl groups. For example, irreducible representations of degenerate affine Hecke algebras correspond to (finite) Cherednik diagrams, and semisimple representations correspond to (finite) skew Young diagrams (see \cite{C,Ro}). This paper proves a double affine analogue of a theorem about affine Hecke algebras due to Guizzi, Nazarov, and Papi.

The result in question appears in \cite{GNP}. It relies on the philosophy that submodules are easier to understand and work with than quotients, (e.g. in computations and examples), and seeks to explicitly realise any irreducible module  $L^{\textrm{aff}}_{\chi}$ for the affine Hecke algebra, normally constructed as a quotient of the Verma module $M^{\textrm{aff}}_{\chi}$, as a submodule of (another) Verma module $M^{\textrm{aff}}_{\tau}$. For a character $\chi_D$ corresponding to a Cherednik diagram $D$, the authors consider the character $\tau=w_0\chi$, for $w_0$ the longest element of the symmetric group, construct a homomorphism $M^{\textrm{aff}}_{\chi} \to M^{\textrm{aff}}_{\tau}$ using rescaled intertwiners, and prove that it factors through the quotient map $M^{\textrm{aff}}_{\chi}\to L^{\textrm{aff}}_{\chi}$, thus realizing $L^{\textrm{aff}}_{\chi}$ as a submodule of $M^{\textrm{aff}}_{\tau}$. Though the existence of such an inclusion could potentially be deduced from general principles (for example, by proving that every Verma module $M^{\textrm{aff}}_{\tau}$ has a simple socle isomorphic to $L_{w_{0}\tau}^{\textrm{aff}}$, or by considering dual modules), the advantage of the construction in \cite{GNP} is in the explicit construction of the eigenvector in $M^{\textrm{aff}}_{\tau}$ with the required eigenvalue, and the underlying combinatorics of fusion of intertwining operators associated to the symmetric group, continuing the work on fusion developed in \cite{C,D,KNP}.

We study the corresponding question for double affine Hecke algebras $\H$, for $n\in \mathbb{N}, n\ge 2$ and a parameter $\kappa \in \mathbb{N}$. The main result of the paper is the following:
\begin{restatable}{theorem}{submodules}
\label{main}
Let $L_D$ be the semisimple irreducible module for the degenerate double affine Hecke algebra $\H$ associated to a periodic skew diagram $D$. Then $L_D$ can be realised as a submodule of a Verma module if and only if $\kappa=1$, or $\kappa\ge 2$ and the diagram $D$ has no infinite column. 
\end{restatable}

Notice some similarities and differences to the main result of \cite{GNP}. Firstly, in the double affine Hecke algebra setting there is no longest element $w_0$ of the Weyl group (affine symmetric group in our case, and symmetric group in \cite{GNP}), so the choice of $\tau$ such that $L_D$ embeds into $M_\tau$ is more involved. Secondly, in the case of degenerate affine Hecke algebras every irreducible module can be embedded into a Verma module, while for degenerate double affine Hecke algebras there exist irreducible modules for which this is not true. These modules are in a sense a degenerate case, and are ``too small'' to be embedded into an induced module. Thirdly, we only prove our result for semisimple modules, as we use the combinatorics described in \cite{SV} to prove that the image of the homomorphism $M_{D} \to M_{\tau}$ is the irreducible module $L_{D}$. We conjecture that a similar result holds for non-semisimple modules as well. Finally, we note that there is no clear way how the double affine result could follow from general principles such as duality, as socles of Verma modules are not always simple. The motivation for studying the question of the existence of the inclusions in the double affine case is thus both in realising simple modules in a more direct and computation-friendly way, and in understanding the combinatorics of the fusion of intertwiners. 

The method of the proof is as follows. We treat the following three cases separately. For $\kappa>1$ and $D$ a periodic skew diagram with an infinite column, we find explicit torsion in the module $L_D$, and prove that submodules of Verma modules have no torsion. For $\kappa>1$ and $D$ a periodic skew diagram with no infinite column, we construct an element $\w$ of the affine Weyl group depending on $D$, and choose $\tau=\w^{-1} \chi_D$. We then construct a homomorphism $F:M_D\to M_\tau$ in a way analogous to \cite{GNP}, using limits of intertwiners. The proof that the homomorphism is well defined relies on the same tools as in the affine case (the combinatorial study of reduced decompositions of elements of reflection groups), but the combinatorics involved is different due to the different choice of the group element $\w$. The proof that $F$ factors through the quotient map $Q:M_D\to L_D$ is combinatorial, and relies on the results in \cite{SV}, as opposed to the algebraic proof of \cite{GNP} for the affine Hecke algebras, which uses a functor to quantum groups and results from their representation theory. Finally, for $\kappa=1$, we find an explicit embedding of any semisimple irreducible module into a Verma module. 

The methods and the results here obtained for the degenerate double affine Hecke algebras apply analogously to the case of double affine Hecke algebras, with the same proofs. 

The roadmap of the paper is as follows. In Section 2, we review the results about degenerate double affine Hecke algebras and their representations which we will use, most notably from \cite{S1,S2,SV}. We also review the corresponding result of \cite{GNP} concerning embeddings of irreducible modules into Verma modules for affine Hecke algebras. In Section 3 we classify the semisimple irreducible modules which cannot be embedded into Verma modules, for $\kappa>1$. In Section 4, for $\kappa>1$, we classify all semisimple irreducible modules which can be embedded into Verma modules, and give an explicit embedding. In Section 5 we deal with the case $\kappa=1$, and find an explicit embedding of any semisimple irreducible module into a Verma module.

\subsection*{Acknowledgements}
I am grateful to Maxim Nazarov for suggesting the problem, and for the many conversations which helped me shape the question. I wish to thank Pavel Etingof for the suggestion to look at the size of modules. I am grateful to Monica Vazirani for a very helpful conversation at the AIM, where she told me about \cite{SV} and her PhD thesis. Paolo Papi, Valerio Toledano Laredo, Vincent van der Noort and James Waldron all read preprints of this paper and improved the exposition with their comments. This work was supported by the EPSRC grant EP/I014071/1 and  the EPSRC grant EP/K025384/1.

\section{Preliminaries}\label{preliminaries}

\subsection{The Weyl group and the affine Weyl group of type $A$} 
All the material in this section is standard and can be found in \cite{H}.

\begin{definition} 
For integer  $n\ge 2$, the extended affine Weyl group of type $A_n$ is the group $\dot{W}_n$ with generators $s_1,\ldots ,s_{n-1}$ and $x_1^{\pm 1},\ldots , x_n^{\pm 1}$ and relations: 
\begin{align*}
s_i^2&=1 \\
s_i s_{i+1} s_i &= s_{i+1} s_i  s_{i+1} \\
s_is_j&= s_js_i \,\,\,\,\,\,\, |i-j|\ne 1 \\
x_ix_j&= x_jx_i \\
s_i x_i &= x_{i+1} s_i \\
s_i x_j &= x_{j} s_i \,\,\,\,\,\,\, j\ne i,i+1.
\end{align*}
\end{definition}

The subgroup generated by $s_1,\ldots ,s_{n-1}$ is the symmetric group $W_n$.
It is the Weyl group of type $A_n$, and it acts on the $\mathfrak{gl}_n$ weight lattice $P=\oplus_{i=1}^n \ZZ \epsilon_i$ by the permutation action $w(\epsilon_i)=\epsilon_{w(i)}$ for $w\in W_n$. The subgroup of $\W$ generated by  $x_1^{\pm 1},\ldots , x_n^{\pm 1}$ is isomorphic to this lattice written multiplicatively. The simple roots for $W_n$ are $\alpha_i=\epsilon_i-\epsilon_{i+1}$ for $i=1,\ldots ,n-1$. The group $\dot{W}_n$ is isomorphic to the semidirect product of $W_n\ltimes P$, and its group algebra is $\CC[\dot{W}_n]=\CC[W_n]\ltimes \CC[x_1^{\pm},\ldots , x_n^{\pm}]$. 

We extend the root lattice to $\oplus_{i=1}^n \ZZ \epsilon_i\oplus \ZZ \mathbf{c}$, and define $\epsilon_{i+kn}=\epsilon_i-k\mathbf{c}$ for $i=1,\ldots ,n$ and $k\in \ZZ$. Then the affine roots $\mathcal{R}$ are $\alpha_{i,j}=\epsilon_i-\epsilon_j$; they satisfy $\alpha_{i+n,j+n}=\alpha_{i,j}$. Positive affine roots can be chosen to be $\mathcal{R}_+=\{\alpha_{i,j}| j>i \}$.

We will use another well known presentation of $\dot{W}_n$. Let $s_{ij}\in W_n$ be the transposition of $i$ and $j$; in particular  $s_i=s_{i,i+1}$. Set $s_0=x_1x_n^{-1}s_{1n}$ and $\pi=x_1s_1s_2\ldots s_{n-1}$. Abusing notation, set $s_{i+kn}=s_i$ for $0\le i\le n-1$, $k\in \mathbb{Z}$. Then $\dot{W}_n$ is generated by $s_0,s_1,\ldots , s_{n-1}, \pi^{\pm 1}$, with the relations: 
\begin{itemize}
\item for any $n:$
\begin{align*}
s_i^2&=1 \\
\pi s_i\pi^{-1} &= s_{i+1};
\end{align*} 
\item for $n\ge 3$, in addition to the above: 
\begin{align*}
s_i s_{i+1} s_i&=s_{i+1} s_i  s_{i+1} \\
s_is_j&=s_js_i \,\,\,\,\, |i-j|\not\equiv 1 \,\, \pmod{n}.
\end{align*}
\end{itemize}

The length function on $\W$ is determined by $l\left(\pi\right)=0$, $l\left(s_i\right)=1$. Let $\W^0$ be the subgroup generated by $s_0,\ldots , s_{n-1}$.  

\begin{comment}
\begin{itemize}
\item If $n=2$: 
\begin{align*}
s_i^2&=1 \\
\pi s_1\pi^{-1} &= s_0 \\
\pi s_0\pi^{-1} &= s_1 
\end{align*} 
\item If $n\ge 3$: 
\begin{align*}
s_i^2&=1 \\
s_i s_{j} s_i&=s_{j} s_i  s_{j}  |i-j|\equiv 1 \,\, \pmod{n}  \\
s_is_j&=s_js_i \,\,\, |i-j|\not\equiv 1 \,\, \pmod{n}  \\
\pi s_i\pi^{-1} &= s_{i+1} \\
\pi s_{n-1}\pi^{-1} &= s_0 
\end{align*}
\end{itemize}
\end{comment}

\subsection{Degenerate double affine Hecke algebra of type $A$}\label{defDAHA}

In the following subsections we recall the definition of degenerate double affine Hecke algebras and some results from \cite{S1,S2,SV}.

\begin{definition}
For integers $n\ge 2$ and $\kappa \ge 1$, the degenerate double affine Hecke algebra (trigonometric Cherednik algebra) of type $A$ is the unital associative algebra $\H$ over $\CC$ such that:
\begin{itemize}
\item[(i)] as a vector space, $\ddot{H}_n\left(\kappa\right)=\CC[\dot{W}_n] \otimes \CC[u_1,\ldots , u_n]$;
\item[(ii)] the natural inclusions $\CC[\dot{W}_n] \hookrightarrow \ddot{H}_n\left(\kappa\right)$ and $\CC[u_1,\ldots , u_n] \hookrightarrow \H$ are algebra homomorphisms;
\item[(iii)] the relations between the generators of $\dot{W}_n$ and $u_1,\ldots , u_n$ are as follows:
\begin{align*}
s_i u_i &= u_{i+1} s_i -1 \,\,\,\,\,\,  i=1\ldots n-1 \\
s_0 u_n &= \left(u_{1}-\kappa\right) s_0 -1 \\
s_i u_j &=  u_j s_i \,\,\,\,\,\, j\not\equiv i,i+1 \, \pmod{n} \\
\pi u_i \pi^{-1}  &=  u_{i+1}  \,\,\,\,\,\,  i=1\ldots n-1  \\ 
\pi u_n \pi^{-1}  &=  u_{1}-\kappa.
\end{align*}
\end{itemize}

The subalgebra of $\ddot{H}_n\left(\kappa\right)$ generated by $\CC[W_n]$ and $\CC[u_1,\ldots , u_n]$ is the degenerate affine Hecke algebra $\dot{H}_n$. 

\end{definition}

While the above definition makes sense for any $\kappa \in \mathbb{C}$, this restriction is common (see \cite{S2,SV}) because the behaviour for $\kappa\in \mathbb{C}\setminus \mathbb{Q}$ is very simple (the appropriately defined category of representations is semisimple), the behaviour for $\kappa\in  \mathbb{Q}\setminus \{0\}$ can be deduced from the behaviour for $\kappa\in \mathbb{N}$, and the behaviour for $\kappa=0$ is very different and usually considered separately. 

The commutator between $u_i$ and $x_j$ can be computed from the above definition as
\begin{equation*}
[u_i,x_j] = \begin{cases}
 \kappa x_i +x_i \sum_{k<i}x_k s_{ki}+\sum_{k>i}x_i s_{ik} &i=j\\
-x_{\min\{ i,j \}}s_{ij} &i\ne j. 
\end{cases}
\end{equation*}
Multiplication in the algebra induces an isomorphism $\CC[x_1^{\pm},\ldots , x_n^{\pm}]\otimes \CC[W_n] \otimes \CC[u_1,\ldots , u_n] \to \ddot{H}_n\left(\kappa\right)$.

It is convenient to define $u_{j+kn}=u_j-k\kappa$ for $j,k\in \ZZ$, $1\le j\le n$. With that convention and the convention $s_{i+kn}=s_{i}$, we can uniformly write the relations as: for all $i,j\in \mathbb{Z}$,
\begin{align*}
s_i u_i &=u_{i+1} s_i -1 \\
s_i u_j &=u_{j} s_i  \quad \textrm {if } j\not\equiv i,i+1 \pmod{n} \\
\pi u_i \pi^{-1}&=u_{i+1}.
\end{align*}

\subsection{An action of $\dot{W}_n$ on $\ZZ$ and on $\CC^n$}\label{notation}

Consider the permutation action modulo $n$ of the group $\dot{W}_n$ on the set $\ZZ$, introduced in \cite{L}, and determined by:
\begin{align*}
s_i \left(j\right)&=\begin{cases}  j+1 & j \equiv i \pmod{n} \\ j-1 & j \equiv i+1 \pmod{n} \\ j & \textrm{otherwise} \end{cases} \\
\pi \left(j\right) &= j+1.
\end{align*}
From this it follows that $$x_i\left(j\right)=\begin{cases}  j+n & j \equiv i \pmod{n} \\  j & \textrm{otherwise.} \end{cases}$$
The action of $\W$ on the set $\mathcal{R}$ of affine roots coincides with the action given by setting $w\left(\epsilon_i\right)=\epsilon_{w(i)}.$

For $\xi=\left(\xi_1,\ldots ,\xi_n\right)\in \CC^n$, we define the functional $\xi:\mathbb{C}[u_1,\ldots, u_n]\to \mathbb{C}$ by $\xi\left(u_i\right)=\xi_i$. Accordingly, extend the indexing of $\xi$ to $\ZZ$  by defining $\xi_{i+kn}=\xi_i-k\kappa$; this is consistent with the above convention $u_{i+kn}=u_i-k\kappa$. Define the action of $\dot{W}_n$ on $\mathbb{C}^n$ by setting $\left(w\xi\right)_i=\xi_{w^{-1}(i)}.$

\subsection{Periodic diagrams}\label{defD}

We now recall the definitions of the combinatorial objects which parametrize irreducible representations of $\H$, periodic Cherednik diagrams and periodic skew Young diagrams. These diagrams are certain sets $D$ of integral points in the plane, generalising Young diagrams. This is one of the possible parametrizations of irreducible representations of various (double) affine Hecke algebras, a version of which is used in \cite{GNP,C,SV,S1,S2}; another one is by Zelevinsky's multisegments, defined in \cite{Z}. As is usual for Young diagrams, we will draw points $\left(a,b\right)\in D\subseteq \mathbb{Z}^2$ as boxes, labelling integer points in the plane like rows and columns of an infinite matrix (rows with labels increasing downward and columns with labels increasing to the right).

\begin{definition}\label{defdiagram}
\begin{enumerate}
\item 
Let $n,m,l$ be integers such that $n\ge 2$, $1\le m\le n$. A \emph{periodic Cherednik diagram} of degree $n$ and period $\left(m,-l\right)$ is a set $D\subseteq \ZZ^2$ such that: 
\begin{itemize}
\item[(i)] restricted to each row, $D$ is a non-empty segment:
$$\forall a \in \ZZ \,\, \exists \, \mu_a \le \lambda_a \textrm{ s.t. } \{b | \left(a,b\right)\in D \}=\left[\mu_a, \lambda_a \right]=\{\mu_a, \mu_a+1,\ldots , \lambda_a \}; $$
\item[(ii)] there is a total of  $n$ points in rows labeled $1$ to $m$:
$$\sum_{a=1}^m \left(\lambda_a-\mu_a+1\right)=n;$$
\item[(iii)] it is periodic of period $\left(m,-l\right)$:
$$D+\ZZ\cdot \left(m,-l\right)=D;$$
\item[(iv)] its left and right edge satisfy:
$$\forall a \,\, \mu_{a+1}\le \mu_a +1 \textrm{ and} $$
$$\textrm{if } \mu_{a+1}= \mu_a +1 \textrm{ then } \lambda_{a+1}\le  \lambda_a +1.$$
\end{itemize}
\item A periodic Cherednik diagram $D$ is called a periodic skew diagram if instead of \emph{(iv)} it satisfies a stronger condition:
\begin{itemize}
\item[(iv')] 
$$\forall a \quad \mu_{a+1}\le \mu_a  \textrm{ and } \lambda_{a+1}\le  \lambda_a.$$
\end{itemize}
\end{enumerate}
\end{definition}

Because of periodicity, specifying a diagram of degree $n$ and period $\left(m,-l\right)$ is equivalent to specifying its first $m$ rows. We call this set the \emph{fundamental domain} of $D$. Given $\mu=\left(\mu_1,\ldots ,\mu_m\right)$ and $\lambda=\left(\lambda_1,\ldots , \lambda_m\right)$, the other endpoints of row segments can be calculated as $\mu_{i+km}=\mu_{i}-kl$, $\lambda_{i+km}=\lambda_{i}-kl$. 

If $D$ is a periodic skew diagram of period $\left(m,-l\right)$, then $l\ge 0$. It is easy to see that if $D$ is a periodic skew diagram and $\left(a,b\right)\in D$, $\left(a+s,b+t\right)\in D$ for some $s,t \ge 0$, then $\left(a+s',b+t'\right) \in D$ for all $0\le s'\le s$, $0\le t'\le t$.

\begin{example}
In the following examples we draw the first $2m$ rows of a diagram. \vspace{0.5cm}

\begin{tabular}{ccc}
Not a periodic Cherednik diagram &A periodic Cherednik diagram & A periodic skew diagram \\
\young(:\hfil\hfil\hfil,:::\hfil,::\hfil\hfill,\hfil\hfil\hfil,::\hfil,:\hfil\hfil)  & \young(\hfil\hfill,:\hfil \hfil,:\hfil\hfil\hfill\hfil,\hfil\hfill,:\hfil \hfil,:\hfil\hfil\hfill\hfil)   &   \young(:::\hfil\hfill,:::\hfil \hfil,::\hfil\hfill\hfil,:\hfil\hfill,:\hfil \hfil,\hfil\hfill\hfil) \\ 
$n=6, m=3, l=1$ & $n=8, m=3, l=0$ & $n=7, m=3, l=2$ \\
$\mu=(1,3,2)$ & $\mu=(1,2,2)$ & $\mu=(2,2,1)$    \\
$\lambda=\left(3,3,3\right)$ & $\lambda=\left(2,3,5\right)$ &  $\lambda=\left(3,3,3\right)$
\end{tabular}

\end{example}

\subsection{Tableaux on periodic diagrams and the content of a tableau}\label{defT}
Next, we label the boxes of $D$ by integers.
\begin{definition}
\begin{itemize}
\item[(i)]A tableau on a periodic Cherednik diagram $D$ of degree $n$ and period $\left(m,-l\right)$ is a bijection $T:D\to \ZZ$, such that for any box $\left(a,b\right)\in D$ $$T\left(\left(a,b\right)+k\left(m,-l\right)\right)=T\left(a,b\right)+kn.$$
\item[(ii)]A tableau is said to be standard  if $T$ is increasing along rows and columns: 
$$\textrm{if } \left(a,b\right),\left(a,b+1\right)\in D \textrm{ then } T\left(a,b\right)<T\left(a,b+1\right),$$
$$\textrm{if } \left(a,b\right),\left(a+1,b\right)\in D \textrm{ then } T\left(a,b\right)<T\left(a+1,b\right).$$
\item[(iii)]A row reading tableau $T_0$ on $D$ is the tableau determined on the first $m$ rows by the condition: for $a=1,\ldots m$ and $\left(a,b\right)\in D$,
$$T_0\left(a,b\right)=\sum_{i=1}^{a-1}\left(\lambda_i-\mu_i+1\right)+b-\mu_a+1.$$ 
\item[(iv)]The content of a tableau $T$ is the function $C_T:\ZZ\to \ZZ$ given by
$$ \textrm{for } i=T\left(a,b\right), \,\,\, C_T(i)=b-a.$$
\item[(v)] For a fixed diagram $D$ we define the action of $\W$ on the set of all tableaux on $D$ by $$\left(wT\right)\left(a,b\right)=w\left(T\left(a,b\right)\right).$$
\end{itemize}
\end{definition}

\begin{example}In the following examples, $n=4$, $m=2$, $l=1$. We place the diagram so that the top left box on this picture is $\left(1,1\right)$, and calculate $\left(C_T(1),C_T\left(2\right),C_T\left(3\right),C_T\left(4\right)\right)$.
\vspace{0.08cm}

\begin{tabular}{ccc}
Periodic tableau & Standard periodic tableau & Row reading tableau \\
\young(:21,:34,65,78)  & \young(:13,:24,57,68)  & \young(:12,:34,56,78)  \\
$\left(1,0,-1,0\right)$ & $\left(0,-1,1,0\right)$ & $\left(0,1,-1,0\right)$\\
\end{tabular}
\end{example}

\subsection{Verma modules, small Verma modules and their irreducible quotients}\label{defML}

Let $D$ be a periodic Cherednik diagram of degree $n$ and period $\left(m,-l\right)$, and $T_0$ its row reading tableau. Assume $m,l$ are such that $\kappa=m+l\ge 1$, and consider the degenerate double affine Hecke algebra $\ddot{H}_n\left(\kappa\right)$. Let $\chi=\chi_D$ be the character of the subalgebra $\CC[u_1,\ldots , u_n]$ determined by:
$$\chi_i=\chi_D\left(u_i\right)=C_{T_0}(i).$$
This is consistent with our conventions $u_{i+kn}=u_i-k\kappa$ and $\chi_{i+kn}=\chi_i-k\kappa$; if $i=T_{0}\left(a,b\right)$, then $i+kn=T_0\left(a+km,b-kl\right)$, so $$C_{T_0}\left(i+kn\right)=\left(b-kl\right)-\left(a+km\right)=\left(b-a\right)-k\left(m+l\right)=C_{T_0}(i)-k\kappa.$$

For any character $\chi$ of $\CC[u_1,\ldots , u_n]$, let $\CC_\chi=\mathbb{C} \mathbf{1}_\chi$ be the one-dimensional representation of $\CC[u_1,\ldots , u_n]$, determined by $u_i \mathbf{1}_\chi=\chi_i \mathbf{1}_\chi$. If $\chi=\chi_D$, we sometimes write $\mathbf{1}_D=\mathbf{1}_\chi$.

\begin{definition}
The standard or Verma module associated to a character $\chi$ of  $\CC[u_1,\ldots , u_n]$ is the $\ddot{H}_n\left(\kappa\right)$ module $$M_{\chi}=\Ind_{\CC[u_1,\ldots , u_n]}^{\ddot{H}_n\left(\kappa\right)} \CC_\chi.$$

If $\chi=\chi_D$ for some diagram $D$, we write $M_D=M_\chi$.
\end{definition}

The module  $M_{\chi}$ is canonically isomorphic to $\CC[\dot{W}_n]$ as a $\CC[\dot{W}_n]$ module, and has a basis $\{ w\mathbf{1}_\chi | w\in \dot{W}_n \}$.

For a diagram $D$ with the row reading tableau $T_0$ and $\chi=\chi_D$, let $I$ be the set of all integers $i \in \{1,\ldots ,n\}$ such that $i$ and $i+1$ are in the same row of $T_0$. Let $W_I\subseteq W_n$ be the subgroup generated by $\{s_i |i\in I\}$. This is the parabolic subgroup of $W_n$, consisting of row preserving permutations of $T_0$. Extend the representation $\CC_{\chi}$ to be a trivial representation of $\CC[W_I]$. This is consistent with the relations of $\H$, and makes $\CC_{\chi}$ into a representation of the subalgebra $\ddot{H}_I\left(\kappa\right)$ of $\H$ generated by $\CC[W_I]$ and $\CC[\u]$. 

\begin{definition}
The small Verma module associated to $D$ is the module $$N_D=\Ind_{\ddot{H}_I\left(\kappa\right)}^{\ddot{H}_n\left(\kappa\right)} \CC_\chi.$$
It is isomorphic to the quotient of $M_D$ by the left $\H$ submodule generated by $\{s_i-1 | i\in I\}$.
\end{definition}

As a $\CC[\dot{W}_n]$ module, $N_D$ is canonically isomorphic to $\CC[\dot{W}_n]/\CC[W_I]$. Abusing notation, write the cyclic generator of $N_D$ which is the image under the quotient morphism of $\mathbf{1}_D\in M_D$  as $\mathbf{1}_D\in N_D$.

\begin{remark}In type $A$, the trigonometric Cherednik algebra is closely related to the rational Cherednik algebra (see \cite{S1}). Define $y_i\in \H$ by the equation $u_i=x_iy_i+\sum_{j<i}s_{ji}$. The subalgebra of $\H$ generated by $x_i,y_i$, $i=1,\ldots, n$ and $s_i$, $i=1,\ldots n-1$ is isomorphic to the rational Cherednik algebra, while the localization of the rational Cherednik algebra at $x_i^{-1}$ recovers the trigonometric Cherednik algebra $\H$.

Consider the Verma module for the rational Cherednik algebra whose lowest weight is the trivial representation of $W_n$. As a vector space this module is isomorphic to $\mathbb{C}[x_1,\ldots, x_n]$, with $x_i$ acting by multiplication, $W_n$ by permutation action, and $y_i$ act by Dunkl operators. On the lowest weight vector, all $y_i$ act as $0$, while the Jucys-Murphy elements act by scalars $0,1,\ldots ,n-1$. Localizing this representation at $x_i^{-1}$, we get the small Verma module for the trigonometric Cherednik algebra associated to the diagram $D$ whose fundamental domain is  $\raisebox{-2pt}{\young(123\hfill \cdots \hfill n)}$. This representation is isomorphic to $\mathbb{C}[x_1^{\pm 1},\ldots, x_n^{\pm 1}]$ as a vector space. The generators $u_i$ respect the natural grading by degree of polynomials, and their eigenvectors in $N_D$ are non-symmetric Jack polynomials. (See \cite{KS,M}).
\end{remark}

\begin{definition}
Let $\mathcal{O}$ be the category of $\H$ modules which are finitely generated, locally finite for the action of $\CC[u_1,\ldots , u_n]$, and such that the generalized eigenvalues for the action of $\left(u_1,\ldots , u_n\right)$ are integers. Let $\mathcal{O}_{ss}$ be the full subcategory of  $\mathcal{O}$ consisting of those modules on which $\left(u_1,\ldots , u_n\right)$ diagonalize. 
\end{definition}

Verma modules and small Verma modules associated to periodic Cherednik diagrams belong to category $\mathcal{O}$. The following theorem describes the irreducible objects in these categories.

\begin{theorem}[\cite{S2,SV}]\label{SVss}
\begin{itemize}
\item[(i)] If $D$ is a periodic Cherednik diagram of degree $n$ and period $\left(m,-l\right)$, $n,m\ge 1$, $\kappa=m+l\ge 1$, then the small Verma module $N_{D}$ for $\H$ has a unique simple quotient. Call this simple module $L_D$. 
\item[(ii)]For any simple module $L$ in category $\mathcal{O}$ of $\H$ representations there exists $1\le m\le n$, and a periodic Cherednik diagram $D$ of degree $n$ and period $\left(m,-\left(\kappa-m\right)\right)$, such that $L$ is isomorphic to $L_D$ as $\H$ modules.
\item[(iii)]The modules $ L_D$ and  $L_{D'}$ are isomorphic if and only if there exists $r\in \ZZ$ such that $D'=D+\left(r,r\right)$.
\item[(iv)]An irreducible module  $ L_D$ is in  $\mathcal{O}_{ss}$ if and only if $D$ is a periodic skew diagram.
\end{itemize}
\end{theorem}

\subsection{Intertwining operators}\label{Intertwiners}
Let us consider the following elements of the appropriate localization of $\H$: 
\begin{align*}
\Phi_i&=s_i+\frac{1}{u_i-u_{i+1}}\\
\Phi_{\pi^{\pm 1}}&=\pi^{\pm 1}.
\end{align*}
In particular, $\Phi_{0}=s_0+\frac{1}{u_0-u_1}=s_0+\frac{1}{-u_1+u_n+\kappa}$, and $\Phi_{i+nk}=\Phi_i$.
They satisfy:
\begin{itemize}
\item for any $n$:
\begin{align*}
\Phi_\pi\Phi_i &= \Phi_{i+1}\Phi_\pi \\
\Phi_i ^2 &=1- \frac{1}{\left(u_i-u_{i+1}\right)^2}; 
\end{align*}
\item for $n\ge 3$, in addition to the above: 
\begin{align*}
\Phi_i\Phi_{i+1}\Phi_i&=\Phi_{i+1}\Phi_{i}\Phi_{i+1}, \\
\Phi_i\Phi_{j}&=\Phi_{j}\Phi_{i}  \,\,\,\, |i-j|\not\equiv1 \pmod{n} ;
\end{align*}
\item for any $n$:
\begin{align*}
\Phi_i u_i&=u_{i+1} \Phi_i \\
\Phi_i u_{i+1}&=u_i \Phi_i \\
\Phi_i u_j&=u_j \Phi_i   \,\,\,\, j\not\equiv i,i+1 \pmod{n}\\
\Phi_\pi u_i&=u_{i+1}\Phi_\pi.
\end{align*}
\end{itemize}

If $w=\pi^rs_{i_1}\ldots s_{i_l}$ is a reduced expression in $\dot{W}_n$, define $\Phi_w=\Phi_\pi^r \Phi_{i_1}\ldots \Phi_{i_l}$; it does not depend on the reduced decomposition. The operators $\Phi_w$ satisfy $\Phi_w u_i \Phi_{w}^{-1}=u_{w(i)}.$ In representations, they act as maps between different eigenspaces of $\u$, and we call them \emph{intertwiners}. Corresponding operators $\Phi_w$ have been considered in \cite{C,Ro}.

Assume that $M$ is a representation of $\H$, $\xi = \left(\xi_1,\ldots ,\xi_n\right)$ is an eigenvalue of $\left(u_1,\ldots ,u_n\right)$, and $M[\xi]=\{v\in M | u_iv=\xi_i v\}$  the corresponding eigenspace. Then:
\begin{itemize}
\item[(i)]$\Phi_{\pi}: M[\xi] \to M[\pi\left(\xi\right)]$  is an isomorphism with the inverse $\Phi_{\pi^{-1}}$;
\item[(ii)]if $\xi_i-\xi_{i+1} \ne 0$, then $\Phi_i|_{M[\xi]}=s_{i}+\frac{1}{\xi_i-\xi_{i+1}}:M[\xi] \to M[s_i\left(\xi\right)]$;
\item[(iii)]if $\xi_i-\xi_{i+1} \ne 0,\pm 1$, then $\Phi_i:M[\xi] \to M[s_i\left(\xi\right)]$  is an isomorphism with inverse $\frac{\left(\xi_i-\xi_{i+1}\right)^2}{1-\left(\xi_i-\xi_{i+1}\right)^2}\Phi_i$.
\end{itemize}

Define also the rescaled intertwiners $\Psi_i$. For an eigenvalue $\xi\in \mathbb{C}^n$ with $\xi_i-\xi_{i+1}\ne 0,\pm 1$, define $$\Psi_i |_{M[\xi]}=\frac{\xi_i-\xi_{i+1}}{\xi_i-\xi_{i+1}+1}\Phi_i=\frac{\xi_i-\xi_{i+1}}{\xi_i-\xi_{i+1}+1}s_i+\frac{1}{\xi_i-\xi_{i+1}+1}.$$ They satisfy $\Psi_i^2=1$, along with the braid relations $\Psi_i\Psi_{i+1}\Psi_i=\Psi_{i+1}\Psi_{i}\Psi_{i+1}, \Psi_i \Psi_j=\Psi_j\Psi_i$ for $|i-j|>1$. This enables us to define $\Psi_w$ for $w\in \W$.

Informally, the usefulness of $\Phi_w$ and $\Psi_w$ comes from the fact that, using $\Phi_w$ instead of $w$ turns the the Hecke algebra relations between $s_i$ and $u_i$ turn into simpler semidirect product relations between $\Phi_i$ and $u_i$. For semisimple representatons of $\H$, where $u_i$ diagonalize and $\Phi_w$ are maps between their joint eigenspaces, understanding the structure of the representation means understanding the combinatorics of the eigenvalues and the action of $\dot{W}_n$ on them.

Notice that for a Cherednik diagram $D$, the small Verma module $N_D$ is the quotient of the Verma module $M_D$ by the submodule generated by $\{\Phi_i | i \in I \}$.

\subsection{Irreducible semisimple modules}\label{resultsSV}

\begin{theorem}[\cite{SV}]\label{defss}
Let $D$ be a periodic skew diagram of degree $n$ and period $\left(m,-l\right)$, $\chi=\chi_D$ the associated character of $\CC[u_1,\ldots, u_n]$, and $M_D$ and $L_D$ the corresponding Verma module and its irreducible quotient for the algebra $\H$. 
\begin{enumerate}
\item The irreducible module $L_D$ is isomorphic to the module with the basis
$$\{ v_T \,\, | \,\, T \textrm{ a standard tableau on D } \}$$
and the action
\begin{align*}
u_i v_T &= C_T(i) v_T \\
\pi v_T &= v_{\pi \left(T\right)} \\
s_i v_T&=\begin{cases}\frac{C_T(i)-C_T\left(i+1\right)+1}{C_T(i)-C_T\left(i+1\right)}v_{s_i\left(T\right)} -\frac{1}{C_T(i)-C_T\left(i+1\right)}v_T, &  s_i\left(T\right) \textrm{ standard}\\ -\frac{1}{C_T(i)-C_T\left(i+1\right)}v_T, &  s_i\left(T\right) \textrm{ not standard} \end{cases}
\end{align*}
\item Identifying $L_D$ with this module, the quotient map $Q: M_D\twoheadrightarrow L_D$ is the unique $\H$  morphism such that $Q\left(\mathbf{1}_D\right)=v_{T_0}$.
\item The kernel of $Q: M_D\twoheadrightarrow L_D$ is generated as an $\H$-module by the set 
$$\{ \Phi_i \Phi_w \mathbf{1}_D | wT_0 \textrm{ standard, } s_iwT_0 \textrm{ not standard}\}.$$

\end{enumerate}
\end{theorem}

Part (1) of this theorem is the section ~4.3. of \cite{SV}, part (2) follows from the definition of $M_D$ as an induced module by Frobenius reciprocity, and part (3) follows from Lemma 3.17 in \cite{SV}. In particular, all eigenspaces of $L_D$ are one dimensional, if $\xi$ is an eigenvalue of $\u$ on $L_D$ then $\xi_i\ne\xi_{i+1}$ for  every $i$, and if $w\xi=\xi$ for some eigenvalue $\xi$ with eigenvector $v\in L_D$ and some $w\ne 1\in \W$, then $\Phi_w v=0$. 

We can describe the module $L_D$ in terms of intertwiners:
\begin{itemize}
 \item if both $T$ and $s_iT$ are standard, 
 $v_{s_iT}= \Psi_i v_T;$
\item if $T$ is standard and $s_iT$ is not, then $s_iv_T$ can be calculated from $\Phi_i v_T=0.$ 
\end{itemize}

We can also describe the quotient map $Q: M_D\twoheadrightarrow L_D$ in terms of intertwiners: 
\begin{itemize}
 \item if $wT_0$ is standard, then using \cite{SV} Lemma 3.17, we get that for some $a_{w'},b_{w'}\in \CC$
$$Q\left(w\One_D\right)= \left(\Psi_w+\sum_{l\left(w'\right)<l\left(w\right)} a_{w'}\Psi_{w'}\right) v_{T_0} = v_{wT_0}+ \sum_{l\left(w'\right)<l\left(w\right)} b_{w'} v_{w' T_0};$$
 \item if $wT_0$ is not standard, then for some $c_{w'}\in \CC$
 $$Q\left(w\One_D\right)=  \sum_{l\left(w'\right)<l\left(w\right)} c_{w'} v_{w' T_0}.$$
\end{itemize}

\begin{lemma}\label{adjacent}
Assume that $D$ is a periodic skew diagram, $T$ a standard tableau on it and $i\in \mathbb{Z}$ is such that $s_iT$ is not standard. Then in $T$, the boxes containing $i$ and $i+1$ are adjacent.
\end{lemma}
\begin{proof}
The tableaux $s_iT$ and $T$ only differ by transposing $i+kn$ and $i+kn+1$ for all $k$. Any integer $z\ne i,i+1$ is bigger than $i$ if and only if it is bigger than $i+1$. Thus, the only way that $s_iT$ can be non-standard while $T$ is standard is that in $T$, $i$ and $i+1$ are comparable, meaning they are in the same row or in the same column. As $T$ is standard, there can be no integer between them, so the only possibilities are:
$$  \begin{ytableau} \mbox{\tiny $i$} &\mbox{\tiny $i+1$} \end{ytableau} \qquad \textrm{   and   } \qquad  \begin{ytableau} \mbox{\tiny $i$} \\ \mbox{\tiny $i+1$}  \end{ytableau}.   $$
\end{proof}

\subsection{Corresponding results for degenerate affine Hecke algebras (\cite{GNP})}

Irreducible representations for the (degenrate) affine Hecke algebra $\dot{H}_n$ are parametrized by finite Cherednik diagrams in a directly analogous way. To a finite Chrednik diagram $D_{fin}$ consisting of $n$ boxes we associate a character $\chi=\chi_{D}$ as the content of the row reading tableaux on $D_{fin}$. The induced module  $M_D^{\textrm{aff}}=M_\chi^{\textrm{aff}}$ is isomorphic to $\CC[W_n]$. Its quotient by $\Phi_i$, for $s_i$ row preserving simple reflections, has a unique an irreducible quotient $L_{D}^{\textrm{aff}}=L_{\chi}^{\textrm{aff}}$. Let $w_0$ be the longest element of $W_n$. 

\begin{theorem}{\cite{GNP}}\label{MaxThm}
For every finite Chrednik diagram $D_{fin}$ there exists a nonzero homomorphism $L_\chi^{\textrm{aff}}\hookrightarrow M_{w_0\chi}^{\textrm{aff}}$.
\end{theorem}

If $\chi_1,\ldots, \chi_n$ are all distinct, then $\Phi_{w_0}\One_{w_0\chi}$ is an eigenvector in $M_{w_0\chi}^{\textrm{aff}}$ with eigenvalue $w_0^2\chi=\chi$, and $\One_{\chi} \mapsto \Phi_{w_0}\One_{w_0\chi}$ determines a homomorphism $M_\chi^{\textrm{aff}}\hookrightarrow M_{w_0\chi}^{\textrm{aff}}$. If some $\chi_i=\chi_j$ for $i\ne j$, then at least one factor of  $\Phi_{w_0}$ has a pole, and $\Phi_{w_0}\One$ is not well defined. The proof replaces $\Phi_{w_0}$ with the limit at $z=\chi_D$ of an appropriate $\CC[W_n]$-valued rational function of $z$, and uses the fusion procedure to show that this function (restricted to an appropriate subset of $\CC^n$) is regular at $z=\chi_D$. This procedure is combinatorial in nature, and relies on the detailed combinatorial study of the possible reduced decompositions of the longest elements $w_0$ of the symmetric group. The goal is to produces an eigenvector $E$ in $M_{w_0\chi}^{\textrm{aff}}$ with the eigenvalue $\chi_D$, which induces a homomorphism $M_\chi^{\textrm{aff}}\hookrightarrow M_{w_0\chi}^{\textrm{aff}}$. To prove that this homomorphism factors through the surjection $M_\chi^{\textrm{aff}}\to L_\chi^{\textrm{aff}}$, the authors apply a functor from representations of the affine Hecke algebra to a certain category of representations of a quantum group, and use know results about corresponding morphisms in that category.

\subsection{The main result}
We now state and prove the corresponding result about inclusions of irreducible modules into Verma modules for double affine Hecke algebras. From now on, we will be concerned with periodic skew Young diagrams and semisimple irreducible representations.

By Theorem \ref{MaxThm}, every irreducible category $\mathcal{O}$ module for degenerate affine Hecke algebras can be realised as a submodule of a Verma module. For double affine Hecke algebras, the answer is more complicated, and there are irreducible modules for which such an inclusion does not exist.  These modules are ``too small", in the following sense: Verma modules are induced, isomorphic as $\CC[\W]$ modules to $\CC[\W],$ and are in particular free $\CC[x_1^{\pm1},\ldots, x_n^{\pm1}]$ modules of rank $n!$. Any submodule of a Verma module is therefore free of $\CC[x_1^{\pm1},\ldots, x_n^{\pm1}]$ torsion. So, any irreducible module which has $\CC[x_1^{\pm1},\ldots, x_n^{\pm1}]$ torsion can not be embedded into a Verma module. We will describe such modules in terms of periodic skew Young diagrams and find their torsion. For all other modules, we will describe a Verma module they inject to and find an explicit embedding. 

\submodules*

We proceed with the proof in three steps. In Section \ref{nocando}, we show how, for $\kappa\ge 2$, an infinite column for $D$ prevents $L_D$ from being embedded into a Verma module (Proposition \ref{one-way}). In Section \ref{yescando}, for $\kappa\ge 2$ and $D$ a diagram with no infinite column, we define an affine Weyl group element $\w$ and a character $\w^{-1}\chi_D$, construct a map of Verma modules $M_D\to M_{\w^{-1}\chi_D}$, and prove it factors through the quotient map $M_D\to L_D$, giving an embedding $L_D\hookrightarrow M_{\w^{-1}\chi_D}$ (Proposition \ref{two-way}, stated at the beginning and proved at the end of Section \ref{yescando}). Finally, in Section \ref{kappa1} we resolve the $\kappa=1$ case, constructing an explicit embedding of every irreducible module into a Verma module in that case (Proposition \ref{thirdcase}).

\section{Irreducible modules which can not be realised as submodules of Verma modules}\label{nocando}
In this section, we prove one part of the main result. Namely, we show:
\begin{proposition}\label{one-way}
Let $D$ be a periodic skew diagram of degree $n$ and period $\left(m,-l\right)$, $\kappa=m+l\ge 2$, and assume that $D$ has an infinite column (there exists $b\in \ZZ$ such that $\left(a,b\right) \in D$ for infinitely many values of $a\in \ZZ$). Then the corresponding irreducible module $L_D$ for the degenerate double affine Hecke algebra $\H$ cannot be embedded into a Verma module $M_\tau$ for any character $\tau$ of $\CC[u_1,\ldots, u_n]$.
\end{proposition}
This is proved through a series of lemmas. The assumption $\kappa\ge 2$ is used to find torsion in Lemma \ref{hastorsion}.

\begin{lemma}\label{sswinfcolumn}
If $D$ is a periodic skew diagram of degree $n$ and period $\left(m,-l\right)$ with an infinite column, then $D$ consists of $k$ consecutive infinite columns. In other words, $l=0$, $n=mk$ for some $k\in \NN$, $\kappa=m$, and there exists $\mu\in \ZZ$, $\lambda=\mu+k-1$ such that $$D=\{ \left(a,b\right) | a\in \ZZ, \mu\le b\le \lambda \}=\ZZ\times [\mu, \lambda].$$
\end{lemma}
\begin{proof}  
Let us first show that if $D$ has an infinite column, then $l=0$. Setting $\mu=\min_{i\in [1,m]} \mu_i$, $\lambda=\max_{i\in [1,m]} \lambda_i$, we see that $D\subseteq [1,m]\times [\mu,\lambda] +\ZZ \left(m,-l\right).$ Assume the column $b$ is infinite, then $\left(a,b\right)\in D$ for all $a\in \mathbb{Z}$. 

For any $r\in \mathbb{Z}$, pick $a\in \mathbb{Z}$ such that  $1+mr \le a\le m+mr$. As $\left(a,b\right)\in D$, it follows that $\left(a,b\right)\in [1,m]\times [\mu,\lambda] +r\left(m,-l\right)$. In particular, $\mu-lr \le b \le \lambda-lr$, and so $\mu-b\le lr\le \lambda -b$ for all $r$. This is only possible if $l=0$. 

Condition (iv') in Definition \ref{defdiagram} now reads: 
$$\mu_1\ge \mu_2\ge \cdots \ge \mu_m\ge \mu_{m+1}=\mu_1-l=\mu_1$$
$$\lambda_1\ge \lambda_2\ge \cdots \ge \lambda_m\ge \lambda_{m+1}=\lambda_1-l=\lambda_1$$
so $\mu_i=\mu$, $\lambda_i=\lambda$ for all $i$, and $D=\ZZ\times [\mu, \lambda]$. Setting $k=\lambda-\mu+1$, we see that the first $m$ rows contain $n=mk$ boxes. 
\end{proof}

\begin{lemma}\label{hastorsion}
Consider the periodic skew diagram $D=\ZZ\times [\mu,\lambda]$ consisting of $k$ infinite columns, with $n=mk$, $l=0$, $\kappa=m$, $k\ge 1$, and $\mu,\lambda\in \ZZ$ with $\lambda=\mu+k-1$. In the associated irreducible representation $L_D$ of $\H$,
$$\left(x_1+x_2+\cdots +x_k\right)v_{T_0}=\left(x_{k+1}+x_{k+2}+\cdots +x_{2k}\right)v_{T_0}=$$
$$=\cdots = \left(x_{\left(m-1\right)k+1}+x_{\left(m-1\right)k+2}+\cdots +x_{mk}\right)v_{T_0}.$$
\end{lemma}
\begin{proof}

Define the following temporary notation: 
$$X_i= 1+s_{ik}+s_{ik+1}s_{ik}+\cdots +s_{ik+k-2}\ldots s_{ik}.$$
Because of the convention $s_{i+n}=s_i$ and $n=mk$, we have $X_{i+m}=X_i$. Furthermore, 
$X_i$ and $s_j$ commute unless $j$ is between $ik-1$ and $ik+k-1$ $\pmod{ n}$, and $X_iX_j=X_jX_i$ for all $i,j$. We will show that 
\begin{equation}
\label{torsion1}
\left(x_1+x_2+\ldots +x_k\right)v_{T_0} =\pi X_0 X_1\ldots X_{m-1} v_{T_0}.
\end{equation}

\begin{comment}
The first two rows of the row reading tableau $T_0$ on $D$ are
\ytableausetup{mathmode, boxsize=2em}
$$\begin{ytableau}
\scriptstyle1 &\scriptstyle 2 &\scriptstyle 3 & \none[\dots] & \scriptstyle k - 1 &\scriptstyle k \\
\scriptstyle k+1 & \scriptstyle k+2 & \scriptstyle k+3 & \none[\dots] & \scriptstyle 2k - 1 &\scriptstyle 2k\\ 
\end{ytableau} $$

The expression for $x_i\in \W$ in terms of $\pi,s_0,\ldots s_{n-1} \in\W$ is $$x_i=\pi  s_{i-2}\ldots s_0 s_{n-1}\ldots s_i.$$

Expressing $x_1,\ldots , x_n \in\W$ in terms of $\pi,s_0,\ldots s_{n-1} \in\W$, we have: 
\begin{align*}
x_1&=\pi s_{n-1}s_{n-2}\ldots s_1 \\
x_2&=s_1 x_1 s_1=\pi s_0 s_{n-1}\ldots s_2 \\
& \vdots & \\
x_i&=\pi  s_{i-2}\ldots s_0 s_{n-1}\ldots s_i \\
& \vdots & \\
x_n &= \pi  s_{n-2}\ldots s_0. 
\end{align*}
\end{comment}

Using the explicit formulas from Theorem \ref{defss}, and the fact that the first row of $T_0$ looks like $\raisebox{-2pt}{\young(123\hfill \cdots \hfill k)}$, we see that $s_1,s_2,\ldots s_{k-1}$ all act on $v_{T_0}\in L_D$ as $1$. From this, using that $x_i=\pi  s_{i-2}\ldots s_0 s_{n-1}\ldots s_i,$ it follows that 
\begin{align*}
x_1v_{T_0}&=\pi s_{n-1}s_{n-2}\ldots s_k v_{T_0},\\ 
x_2v_{T_0}&=\pi s_0 s_{n-1}s_{n-2}\ldots s_k v_{T_0},\\ 
&\vdots& \\
x_kv_{T_0}&=\pi s_{k-2}\ldots s_0 s_{n-1}\ldots s_k v_{T_0}.
\end{align*}
Summing these, we get that
\begin{equation}\label{torsion2} 
\left(x_1+\ldots +x_k\right)v_{T_0} = \pi X_0 s_{n-1}s_{n-2}\ldots s_k v_{T_0}.
\end{equation}

Let us now calculate $s_{2k-1}\ldots s_{k}v_{T_0}$. In what follows, we use the convention that all products have decreasing indices, and that the product over the empty set is $1$. The first two rows of $T_0$ are 
$$T_0=
{\Large \begin{ytableau}
{\scriptstyle 1} & \scriptstyle 2 &\scriptstyle 3 & \none[\dots] & \scriptstyle k - 1 &\scriptstyle k \\
\scriptstyle k+1 & \scriptstyle k+2 & \scriptstyle k+3 & \none[\dots] & \scriptstyle {2k - 1} &\scriptstyle 2k
\end{ytableau} },
$$
and we see that the tableaux $s_kT_0$, $s_{k+1}s_kT_0$, \ldots , $s_{2k-2}\ldots s_kT_0$ are standard, while $s_{2k-1}\ldots s_kT_0$ is not. Using explicit formulas from Theorem \ref{defss}, we can prove by induction on $i$ that for all $i=0,\ldots , k-2$, 
$$s_{k+i}  \ldots s_kv_{T_0}=\left(\prod_{a=0}^{i}\Psi_{k+a}-\sum_{j=0}^{i} \frac{1}{k-j} \prod_{a=0}^{j-1}\Psi_{k+a} \right) v_{T_0}.$$
From this we get
\begin{align*}
s_{2k-1}  \ldots s_kv_{T_0}&=\left(-\prod_{a=0}^{k-2}\Psi_{k+a}-\sum_{j=0}^{k-2} \frac{1}{k-j} \prod_{a=0}^{j-1}\Psi_{k+a} \right) v_{T_0}\\
&= \sum_{j=0}^{k-1} \frac{-1}{k-j} \prod_{a=0}^{j-1}\Psi_{k+a}   \,\, v_{T_0}\\
&= \sum_{j=0}^{k-1} \left(-1+\sum_{i=j+1}^{k-1} \frac{1}{k-j}  \right)\prod_{a=0}^{j-1}\Psi_{k+a}   \,\, v_{T_0}\\
&= \left(-\sum_{i=0}^{k-1} \prod_{a=0}^{i-1}\Psi_{k+a} +\sum_{i=1}^{k-1}\sum_{j=0}^{i-1} \frac{1}{k-j}  \prod_{a=0}^{j-1}\Psi_{k+a} \right)  v_{T_0}\\
&= \sum_{i=0}^{k-1} \left(-\prod_{a=0}^{i-1}\Psi_{k+a} + \sum_{j=0}^{i-1} \frac{1}{k-j}  \prod_{a=0}^{j-1}\Psi_{k+a} \right)  v_{T_0}\\
&= \sum_{i=0}^{k-1} \left(-\prod_{a=0}^{i-1}s_{k+a} \right)  v_{T_0}\\
&= -X_1 v_{T_0}.
\end{align*}

Going back to (\ref{torsion2}), we can conclude that: 
\begin{align*}
\left(x_1+\cdots +x_k\right)v_{T_0} &= \pi X_0 s_{n-1}s_{n-2}\ldots s_k v_{T_0} \\
 &= \pi X_0 s_{n-1}s_{n-2}\ldots s_{2k}X_1 v_{T_0} \\
 &= \pi X_0 X_1  s_{n-1}s_{n-2}\ldots s_{2k}v_{T_0} \\
& \cdots \\
&= \pi X_0 X_1\ldots X_{m-1}v_{T_0}.
\end{align*}

An analogous calculation then shows that
\begin{align*}
\left(x_{k+1}+x_{k+2}+\cdots +x_{2k}\right)v_{T_0}&=\pi X_1 X_2\ldots X_{m-1} X_{m} v_{T_0} \\
&=\pi X_1 X_2\ldots X_{m-1}X_{0} v_{T_0} \\
&= \pi X_0 X_1\ldots X_{m-2} X_{m-1} v_{T_0},
\end{align*}
thus proving that $\left(x_1+x_2+\cdots +x_k\right)v_{T_0}=\left(x_{k+1}+x_{k+2}+\cdots +x_{2k}\right)v_{T_0} $. 
\end{proof}

\begin{example}
A very similar statement should hold for periodic Cherednik diagrams. Here are some examples: 

\begin{enumerate}
\item 
$\young(12,:3)$, $\left(m,-l)=(2,0\right)$:
$x_1+x_2=x_3=\pi \left(-\Phi_0-\frac{1}{2}\right)$

\item 
$\young(1,23)$, $\left(m,-l\right)=\left(2,0\right)$: 
$x_1^{-1}=x_2^{-1}+x_3^{-1}=\pi^{-1}\left(-\Phi_0-\frac{1}{2}\right)$

\item 
$\young(1,234)$, $(m,-l)=(2,0)$:
: $x_1^{-1}=x_2^{-1}+x_3^{-1}+x_4^{-1}=\pi ^{-1}\left(-\Phi_3\Phi_0-\frac{1}{2}\Phi_0-\frac{1}{3}\right)$

\item 
$\young(123,:45,::6)$, $\left(m,-l\right)=\left(3,0\right)$: $x_1+x_2+x_3=x_4+x_5=x_6=\left(\Phi_3+\frac{1}{2}\right)\left(\Phi_1\Phi_0+\frac{1}{2}\Phi_0+\frac{1}{3}\right)$

\item 
$\young(1,:2)$, $\left(m,-l\right)=\left(2,0\right)$: $L_D=M_D$.

\end{enumerate}
\end{example}

\begin{lemma}\label{canthavenotorsion}
If an irreducible module $L_D$ can be embedded into a Verma module $M_\tau$, then the subalgebra $\CC[x_1^{\pm 1},\ldots , x_n^{\pm 1}]$ of $\H$ acts on the cyclic vector $v_{T_0} \in L_D$ without torsion. 
\end{lemma}
\begin{proof}
As $\CC[\W]$-modules, Verma modules are isomorphic to $\CC[\W]\cong \CC[W_n]\ltimes \CC[x_1^{\pm 1},\ldots , x_n^{\pm 1}]$. In particular, they are free $\CC[x_1^{\pm 1},\ldots , x_n^{\pm 1}]$-modules of rank $|W_n|=n!$. The action of $\CC[x_1^{\pm 1},\ldots , x_n^{\pm 1}]$ on the cyclic vector $v_{T_0}$ of $L_D$ produces a $\CC[x_1^{\pm 1},\ldots , x_n^{\pm 1}]$-submodule of $L_D$, so we have a sequence of inclusions of $\CC[x_1^{\pm 1},\ldots , x_n^{\pm 1}]$-modules $$\CC[x_1^{\pm 1},\ldots , x_n^{\pm 1}]v_{0}\hookrightarrow L_D \hookrightarrow M_\tau\cong \CC[W_n]\ltimes \CC[x_1^{\pm 1},\ldots, x_n^{\pm 1}] .$$
Thus, $\CC[x_1^{\pm 1},\ldots , x_n^{\pm 1}]v_{T_0}$ is a submodule of a free $\CC[x_1^{\pm 1},\ldots , x_n^{\pm 1}]$-module, so it is free of torsion. 
\end{proof}

\begin{proof}[Proof of Proposition \ref{one-way}]
If $D$ is a periodic skew diagram with an infinite column, then by Lemma ~\ref{sswinfcolumn} it consists of $k$ consecutive infinite columns, and $n=mk$. By Lemma ~\ref{hastorsion}, $\sum_{i=1}^k x_{i}$ and $\sum_{i=1}^k x_{k+i}$ act the same on the cyclic vector $v_{T_0}$. If $m=\kappa \ge 2$, then $\sum_{i=1}^k x_{i}\ne \sum_{i=1}^k x_{k+i}$, so $v_{T_0}$ is a torsion element of $L_D$. By Lemma ~\ref{canthavenotorsion}, $L_D$ cannot be embedded into a Verma module. 

This is the only place where the assumption $\kappa\ge 2$ was used. 
\end{proof}

\section{Irreducible modules which can be realised as submodules of Verma modules, $\kappa \ge 2$ case}\label{yescando}

In this section we prove another part of the main theorem:
\begin{proposition}\label{two-way}
Let $D$ be a periodic skew diagram of degree $n$ and period $\left(m,-l\right)$, $\kappa=m+l\ge 2$, and assume that $D$ has no infinite column. Then the corresponding irreducible module $L_D$ for the degenerate double affine Hecke algebra $\H$ can be embedded into a Verma module.
\end{proposition}

For this whole section, let $D$ be a fixed periodic skew diagram satisfying the assumptions of the proposition. 
\begin{comment}
We will first construct a permutation $\w\in \W$ of $\ZZ$, and fix a particular reduced decomposition of $\w$ into simple reflections. We then use this decomposition to construct a vector $\Phi_{\w} \One_{\w^{-1}\chi_D}\in M_{\w^{-1}\chi_D}$. Depending on the diagram $$D, some factors of the intertwiner $\Phi_{\w}$ can have a pole, and we use the fusion procedure developed in \cite{C,GNP,KNP} to resolve them. This gives us a nonzero eigenvector for $\u$ with eigenvalue $\chi_D$ in $M_{\w^{-1}\chi_D}$, which induces a map of Verma modules $M_{\chi_D}\to M_{\w^{-1}\chi_D}$. Finally, we show that it factors through the irreducible quotient $L_D$ of $M_{\chi_D}$, realizing $L_D$ as a submodule of $M_{\w^{-1}\chi_D}$. 

The main differences from \cite{GNP} are that the permutation $\w$ depends on the periodic skew diagram $D$, while in the case of affine Hecke algebras one can uniformly use the longest element of the Weyl group $W_n$ \left(the notion of the longest element does not make sense in the infinite affine Weyl group $\W$\right). We use similar ideas to prove the existence of the eigenvector, but the combinatorics is different. The proof that the map of Verma modules factors through the irreducible quotient of $M_{\chi_D}$ is also different in spirit; here it relies on explicit combinatorial description of the kernel of the quotient map $M_{\chi_D}\to L_D$.
\end{comment}

\subsection{Construction of a permutation $\w$}\label{constructw}
Let $\left(a_1,b_1\right)\in D$ be an arbitrary box, and let us start by assigning a permutation $\w_{\left(a_1,b_1\right)}$ of $\ZZ$ to this choice. Informally, it is given by reading the numbers in the tableau $T_0$ on $D$, starting at $1\mapsto T_0\left(a_1,b_1\right)$, reading up each column until its end, and then continuing up the next column to the left. More formally, $\w_{\left(a_1,b_1\right)}$ is the unique permutation of $\mathbb{Z}$ satisfying the following properties: 
\begin{enumerate}
\item $\w_{\left(a_1,b_1\right)}(1)=T_{0}\left(a_1,b_1\right)$;
\item if $\w_{\left(a_1,b_1\right)}(i)=T_0\left(a,b\right)$, and the box $\left(a-1,b\right)$ immediately above $\left(a,b\right)$ is in $D$, then $\w_{\left(a_1,b_1\right)}\left(i+1\right)=T_0\left(a-1,b\right)$;
\item if $\w_{\left(a_1,b_1\right)}(i)=T_0\left(a,b\right)$ and the box $\left(a-1,b\right)$ is not in $D$, then $i+1$ maps to the bottom entry of the next nonempty column to the left of column $b$.
\end{enumerate}

\begin{example}\label{example4}
Let $D$ be the periodic tableau with the fundamental domain $\raisebox{-5pt}{\young(12,34)}$ and the period $\left(2,-1\right)$, placed so that $1=T_0\left(1,1\right)$. The row reading tableau on it looks like 
$$\begin{ytableau}
\none[\vdots] \\
\none & \none &-3&-2\\
\none & \none & -1&0\\
\none & 1& 2\\
\none & 3&4 \\
5&6\\
7&8\\
\none[\vdots] \\
\end{ytableau} $$
Examples of the permutations defined above are
$$\w_{\left(4,1\right)}=\left( \begin{array}{cccccccccccccc}
\ldots &  -3&-2&-1&0& 1 & 2 & 3 & 4 &5&6&7&8 &\ldots\\
    &4 & 2 & -1 & -3 & 8 & 6 & 3 & 1& 12 &10  & 7&5&
  \end{array} \right) $$
$$\w_{\left(2,2\right)}=\left( \begin{array}{cccccccccccccc}
\ldots &  -3&-2&-1&0& 1 & 2 & 3 & 4 &5&6&7&8 &\ldots\\
&  0&-2& -5& -7&4 & 2 & -1 & -3 & 8 & 6 & 3 & 1&
  \end{array} \right).$$  
\end{example}

\begin{lemma}\label{defw}
\begin{enumerate}
\item For every choice of $\left(a_1,b_1\right)\in D$, the permutation $\w_{\left(a_1,b_1\right)}$ of $\ZZ$ satisfies $\w_{\left(a_1,b_1\right)}	\left(i+n\right)=\w_{\left(a_1,b_1\right)}(i)+n$.
\item The permutation $\w_{\left(a_1,b_1\right)}$ is given by the action from Section \ref{notation} of a unique element of $\W$. We call this element $\w_{\left(a_1,b_1\right)}$ as well. 
\item For any $\left(a_1,b_1\right), \left(a_2,b_2\right) \in D$, permutations $\w_{\left(a_1,b_1\right)}$ and $\w_{\left(a_2,b_2\right)}$ differ by right multiplication by some power of $\pi\in \W$.
\item There is a unique choice of $\left(a_1,b_1\right)\in D$ such that $\w_{\left(a_1,b_1\right)}\in \W$ is in a subgroup $\W^0$ generated by $s_0,\ldots s_{n-1}$ (its reduced decomposition does not contain a power of $\pi$). For such a choice, let $\w=\w_{\left(a_1,b_1\right)}$.
\end{enumerate}
\end{lemma}
\begin{proof}
\begin{enumerate}
\item  The period of the diagram $D$ is $\left(m,-l\right)$, and the tableau $T_0$  is a bijection $T_0:D\to \ZZ$ satisfying $T_0\left(a+m,b-l\right)=T_0\left(a,b\right)+n$. Thus, the union of any $l$ consecutive columns is a fundamental domain for $D$, and the set of values of $T_0$ on the boxes in these $l$ columns is a set of $n$ numbers, one from each congruence class of integers modulo $n$. When constructing $\w_{\left(a_1,b_1\right)}$ by reading the values of $T_0$ up each column, moving from right to left in columns, there are $n-1$ boxes to be read between $\w_{\left(a_1,b_1\right)}(i)=T_0\left(a,b\right)$ and $T_0\left(a+m,b-l\right)$. Thus, $$\w_{\left(a_1,b_1\right)}\left(i+n\right)=T_0\left(a+m,b-l\right)=T_0\left(a,b\right)+n=\w_{\left(a_1,b_1\right)}(i)+n.$$
\item From (1) it follows that $\w_{\left(a_1,b_1\right)}$ is completely determined by $\w_{\left(a_1,b_1\right)}(1), \ldots ,\w_{\left(a_1,b_1\right)}\left(n\right)$, which are integers from different congruence classes modulo $n$. Using the appropriate powers of $x_1,\ldots, x_n$, 
we can get a permutation $x_1^{k_1}\ldots , x_n^{k_n}\w_{\left(a_1,b_1\right)}$, which is also periodic and 
which maps $\{1,\ldots ,n \}$ to itself. This permutation is given by the action \ref{notation} an element $w$ of the symmetric group. Hence, $\w_{\left(a_1,b_1\right)}$ is a permutation of $\mathbb{Z}$ given by the action of the element $x_1^{-k_1}\ldots , x_n^{-k_n}w\in \W$. 
\item If $\w_{\left(a_1,b_1\right)}$ and $\w_{\left(a_2,b_2\right)}$ are both obtained in this way from a diagram $D$, for different choices of starting boxes, then they are two permutations of $\ZZ$ obtained by reading all the integers in $T_0$ in the same order, but with a different starting point $\w_{\left(a_1,b_1\right)}(1)\ne \w_{\left(a_2,b_2\right)}(1)$. Thus, one can be obtained from another by precomposing it with a shift of all integers by some fixed $k\in \ZZ$, $\w_{\left(a_2,b_2\right)}=\w_{\left(a_1,b_1\right)}\pi^k$.
\item The action of $\W$ is defined so that $\sum_{j=1}^n s_i\left(j\right)=\sum_{j=1}^n j$, for all $i=0,\ldots n-1,$ while $\sum_{j=1}^n\pi\left(j\right)=\sum_{j=1}^n\left(j+1\right).$ Choose some $\w_{\left(a_1,b_1\right)}$, and precompose it with the appropriate power of $\pi^k$ to get $\w_{\left(a_1,b_1\right)}\pi^k$ which satisfies $\sum_{j=1}^n  \w_{\left(a_1,b_1\right)}\pi^k\left(j\right)=\sum_{j=1}^n j$. This permutation is then in the subgroup $\W^0$, and it is of the form $\w_{\left(a_1,b_1\right)}\pi^k=\w_{\left(a_2,b_2\right)}$, for $\left(a_2,b_2\right)=T_0^{-1}\left(\w_{\left(a_1,b_1\right)}\left(k\right)\right)$.
\end{enumerate}
\end{proof}

From now on, for a given $D$, we will consider $\w\in \W^0$ as above. The choice of $\w$ over $\w\pi^k$ for some $k\in \ZZ$ will ease the notation in the proof, but is not significant. The modules $M_{\w^{-1}\chi_D}$ and $M_{\pi^{-k}\w^{-1} \chi_D}$ are isomorphic, with the isomorphism determined by $G\left(\One_{\w^{-1}\chi_D}\right)=\Phi_{\pi^{k}}\One_{\pi^{-k}\w^{-1}\chi_D}$. If $F:L_D\hookrightarrow M_{\w^{-1}\chi_D}$ is a homomorphism of modules, then so is $G\circ F: L_D\hookrightarrow M_{\pi^{-k}\w^{-1}\chi_D}$.

\subsection{Convex orders on roots and reduced decompositions}
We want to use the intertwiner $\Phi_{\w}$ to define an inclusion of an irreducible module into a Verma module, and this intertwiner is defined using a reduced decomposition of $\w$ into simple reflections. Some of its factors might have poles. To be able to calculate $\Phi_{\w}$ in those cases as well, we need to fix a particular reduced decomposition of $\w$, which will allow us to use the fusion procedure from \cite{GNP} to resolve the poles. In this subsection, we recall results about reduced decompositions in reflection groups.

Let $w$ be an arbitrary element of the group $\W^0$, and let $\mathcal{I}_w=\mathcal{R}_{w^{-1}}=\mathcal{R}_+\cap w\left(-\mathcal{R}_+\right)$ be the set of inversions of $w^{-1}$. Let $w=s_{i_1}\ldots s_{i_l}$ be a reduced decomposition of $w$. Consider the $l$ roots given by $\beta_k=s_{i_1}\ldots s_{i_{k-1}}\alpha_{i_k}$. Then $\mathcal{I}_w=\{ \beta_1, \ldots, \beta_l\}$, and in particular, the collection of $\beta_k$ does not depend on the decomposition. Setting $\beta_1<\beta_2<\ldots < \beta_l$ gives an order on $\mathcal{I}_w$ which does depend on the decomposition; in fact, it completely determines it. 

\begin{definition}[\cite{P1}]
A total order $<$ on $\mathcal{I}_w$ is called \emph{convex} if it satisfies: 
\begin{enumerate}
\item If $\alpha,\beta,\alpha+\beta, \in \mathcal{I}_w$, then $\alpha+\beta$ is between $\alpha$ and $\beta$; 
\item If $\alpha+\beta, \beta \in \mathcal{I}_w$ and $\alpha \in \mathcal{R}_+\setminus \mathcal{I}_w$, then $\beta< \alpha+\beta$.
\end{enumerate}
\end{definition}

\begin{lemma}[\cite{GP,P1}]\label{convex}
\begin{enumerate}
\item Associating a total order on $\mathcal{I}_w$ to every reduced decomposition of $w$ as above is a bijection between reduced decompositions of $w$ and convex orders on $\mathcal{I}_w$.
\item If $\alpha_{i,j-1},\alpha_{i,j},\alpha_{j-1,j}\in \mathcal{I}_w$, and $<$ is a convex order on it such that $\alpha_{i,j-1}<\alpha_{i,j}$ are adjacent, then there exists a convex order $<'$ on $\mathcal{I}_w$ such that $<$ and $<'$ are the same on all elements of $\mathcal{I}_w$ smaller than $\alpha_{i,j}$, and $\alpha_{i,j-1}<'\alpha_{i,j}<'\alpha_{j-1,j}$ are adjacent in $<'$. An analogous claim holds for $>$. 
\end{enumerate}
\end{lemma}

\subsection{A convex order on $\mathcal{I}_{\w}$ and a reduced decomposition of $\w$} \label{specialorder}
In this section, we fix a particular convex order on the set $\mathcal{I}_{\w}$ associated to the element $\w\in \W$  defined in Section \ref{constructw}. We will use the bijection between reduced decompositions of $\w$ and convex orders on $\mathcal{I}_{\w}$, and the fact that the set of inversions $\mathcal{I}_{\w}$ can be at once recovered from the form of $\w$ as a permutation of $\ZZ$.

To do that, we will represent the affine roots $\alpha_{ij}=\epsilon_i-\epsilon_j$ graphically as arrows on the row reading tableau $T_0$ on the periodic skew diagram $D$ we are working with. Draw $\alpha_{i,j}$ as an arrow 
$\overrightarrow{ij}$ on the diagram from the box $T_0^{-1}(i)$ to the box $T_0^{-1}\left(j\right)$. The roots satisfy $\alpha_{i,j}=\alpha_{i+n,j+n}$ and $T_0$ satisfies $T_0\left(a+m,b-l\right)=T_0\left(a,b\right)+n$, so any root can be represented in this way by countably many arrows, differing by shifts by a multiple of the period $\left(m,-l\right)$. If we fix the fundamental domain of $D$ consisting of the first $m$ rows, any such root can be represented uniquely by an arrow finishing in this fundamental domain. The only exception to this are the roots $k\mathbf{c}$, which can be written as $\epsilon_{i}-\epsilon_{i+kn}$ for any $i$, but they will not be relevant to us (in terms of the usual root theory for Kac-Moody algebras, this means considering only real roots).

The tableau $T_0$ is row reading, and the set of positive roots is $\{ \alpha_{i,j} | i<j\}$, so arrows $\overrightarrow{ij}$ associated to positive roots are:
\begin{enumerate} 
\item pointing right, if $i$ and $j$ are in the same row;
\item or pointing down, or down and left, or down and right, if $i$ is in some row above $j$.  
\end{enumerate}
In particular, because of $l>0$, the roots $k\mathbf{c}$ are represented by arrow pointing down and left.

\begin{lemma}\label{Iw}
The set $\mathcal{I}_{\w}$ is the finite set of all roots $\alpha_{ij}$ represented on $T_0$ by arrows pointing right, pointing down, and pointing down and right.  
\end{lemma}
\begin{proof}
A positive affine root $\alpha_{ij}$ with $i<j$ is in $\w(-\mathcal{R}_+)$ if and only if $\w^{-1}(i)>\w^{-1}\left(j\right)$. This means that, when reading off the integers for $\w$ up each column and taking columns in order from right to left, the integer $i$ is read after $j$. So either $i$ is in a column to the left of $j$ (so $\overrightarrow{ij}$ points right or right and down), or $i$ and $j$ are in the same column with $i$ above $j$ (so $\overrightarrow{ij}$ points down).
\end{proof}

Now let us fix a convex order on $\mathcal{I}_{\w}$. 
\begin{definition}\label{defspecial}
For vectors $\alpha_{i,j},\alpha_{p,q}\in \mathcal{I}_{\w}$ written so that $1\le j,q \le n$, define
$$\alpha_{i,j}<\alpha_{p,q} \,\,\,\, \textrm{if} \,\,\,\, \begin{array}{l} \w^{-1}(i)>\w^{-1}\left(p\right) \textrm{ or } \\ i=p \textrm{ and } j<q.\end{array} $$
We refer to this order as \emph{special}.
\end{definition} 

In words, we order the roots $\alpha_{ij}$ with a version of lexicographical order, first by the first index $i$, and then by the second index $j$, where the order on the second index is standard order on $\ZZ$, and order in the first index is the opposite of the way we calculate $\w$: we read the numbers down each column and taking columns in order from left to right. 

\begin{example}\label{example4c}\label{example4b}
Consider the diagram from example \ref{example4}, and 
$$\w=\w_{\left(0,2\right)}=\left( \begin{array}{cccccccccccccc}
\ldots &  -3&-2&-1&0& 1 & 2 & 3 & 4 &5&6&7&8 &\ldots\\
 & -5 & -7 &4 & 2 & -1 & -3 & 8 & 6 & 3 & 1& 12 & 10 &
  \end{array} \right).$$  
Then
$$\mathcal{I}_{\w}=\{\alpha_{1,2},\alpha_{1,3},\alpha_{1,4},\alpha_{-3,2},\alpha_{-3,4},\alpha_{3,4},\alpha_{-1,2},\alpha_{-1,4},\alpha_{2,4}\}.$$
The order on the first index is $1<3<-3<-1<2$, so the special order on $\mathcal{I}_{\w}$ is
$$\alpha_{12}<\alpha_{13}<\alpha_{14}<\alpha_{34}<\alpha_{-32}<\alpha_{-34}<\alpha_{-12}<\alpha_{-14}<\alpha_{24}.$$
\end{example}

\begin{lemma}\label{special}
The special order on $\mathcal{I}_{\w}$ from Definition \ref{defspecial} satisfies: 
\begin{enumerate}
\item If $\alpha,\beta,\alpha+\beta \in \mathcal{I}_{\w}$, then $\alpha+\beta$ is between $\alpha$ and $\beta$.
\item If $\alpha+\beta, \beta \in \mathcal{I}_{\w}$ and $\alpha \in \mathcal{R}_+\setminus \mathcal{I}_{\w}$, then $\beta< \alpha+\beta$.
\item If $i=T_0\left(a,b\right)$ and $j=T_0\left(a+c,b+c\right)$ are in different boxes with the same content and $i<j$, then the roots $\alpha_{i,j-1}$, $\alpha_{i,j}$ and $\alpha_{j-1,j}$ are in $\mathcal{I}_w$, and $\alpha_{i,j-1}<\alpha_{i,j}$ are adjacent in the special order.
\end{enumerate}
\end{lemma}
\begin{proof}
\begin{enumerate}
\item This statement is symmetric with respect to $\alpha$ and $\beta$, so we may assume that $\alpha=\alpha_{ij}$, $\beta=\alpha_{jk}$ and $\alpha+\beta=\alpha_{ik}$, for some $i<j<k \in \ZZ$. By using $\alpha_{pq}=\alpha_{p+n,q+n}$ we can also assume that $1\le k\le n$.

The root $\alpha=\alpha_{ij}$ is in $\mathcal{I}_{\w}$, which means that $\w^{-1}(i)>\w^{-1}\left(j\right)$, and consequently that $\alpha+\beta=\alpha_{ik}<\alpha_{jk}=\beta$ in the special order. 

If $j\ge 1$, then it also satisfies $1\le j\le n$ and the root $\alpha=\alpha_{ij}$ is written so that it ends in the correct fundamental domain. By the definition of special order and using $j<k$, we get $\alpha=\alpha_{ij}<\alpha_{ik}=\alpha+\beta$.

If $j<1$, then let $p\in \NN$ be such that $1\le j+pn\le n$, and the correct way to write $\alpha$ is as $\alpha_{i+pn,j+pn}$. It follows that
$\w^{-1}\left(i+pn\right)=\w^{-1}(i)+pn>\w^{-1}(i)$, so $\alpha=\alpha_{i+pn,j+pn}<\alpha_{i,k}=\alpha+\beta$. 

So, in either case, $\alpha<\alpha+\beta<\beta$ in the special order. 

\item This statement is not symmetric with respect to $\alpha$ and $\beta$, so let us consider two cases.

CASE 1. $\alpha=\alpha_{ij}$, $\beta=\alpha_{jk}$. We may assume $1\le k\le n$. As $\alpha \notin \mathcal{I}_{\w}$, it follows that $\w^{-1}(i)<\w^{-1}\left(j\right)$ and consequently $\alpha+\beta=\alpha_{ik}>\alpha_{jk}=\beta$ in the special order. 

CASE 2. $\alpha=\alpha_{jk}$, $\beta=\alpha_{ij}$, and we assume that  $1\le k\le n$. If $1\le j$ as well, then $j<k$ implies $\beta=\alpha_{ij}<\alpha_{ik}=\alpha+\beta$. If $j<1$, then for $p>0$ such that $1\le j+pn\le n$ we have $\w^{-1}\left(i+pn\right)=\w^{-1}(i)+pn>\w^{-1}(i)$ so $\beta=\alpha_{i+pn,j+pn}<\alpha_{ik}=\alpha+\beta$.

\item  If $i<j$ are in different boxes with the same content, then $i=T_0\left(a,b\right)$ and $j=T_0\left(a+c,b+c\right)$ for some $c>0$. As $D$ is a periodic skew diagram, the box $\left(a+c,b+c-1\right)$ is also in $D$, 
and $T_0\left(a+c,b+c-1\right)=j-1$. Assume without loss of generality that $1\le j\le n$. As $j-1$ is in the same row as $j$, it also satisfies $1\le j-1\le n$, so both $\alpha_{i,j-1}$ and $\alpha_{i,j}$ are correctly written positive roots. The arrow $\overrightarrow{j-1,j}$ is pointing right, so $\alpha_{j-1,j}\in \mathcal{I}_{\w}$. If $c>1$ then $j-1$ is to the right of $i$, and $\overrightarrow{i,j-1}$ is pointing down and right; if $c=1$ then $i$ and $j-1$ are in the same column $\overrightarrow{i,j-1}$ is pointing down. In either case, $\alpha_{i,j-1}\in \mathcal{I}_{\w}$. In the special order, $\alpha_{i,j-1}<\alpha_{i,j}$ are adjacent.
\end{enumerate}
\end{proof}

We conclude by Lemma \ref{special} that the  special order is convex, 
and it determines a reduced decomposition of $\w$ by Lemma \ref{convex} .

\begin{example}\label{example4d}
The special order from Example \ref{example4c} corresponds to the reduced decomposition
$$\w=s_1s_2s_3s_2s_0s_1s_3s_0s_3.$$
\end{example}

\begin{remark}
One can associate an $l$-tuple (ordered multiset) of roots to any decomposition of $w\in \W^0$ of length $l$ into simple reflections, using the same formulas for $\beta_i$. The decomposition is non-reduced if and only if there is a root $\beta$ such that both $\beta$ and $-\beta$ appear in the multiset. In this case, the product of simple reflections can be reordered so that the two instances of the repeating root are adjacent, and the corresponding factors cancel. 
\end{remark}

\begin{example}
Consider the non-reduced product $s_1s_2s_1s_2$. The roots associated to them are
$\beta_1=\alpha_{12}$, $\beta_2=s_1\left(\alpha_{23}\right)=\alpha_{13}$, $\beta_3=s_1s_2\left(\alpha_{12}\right)=\alpha_{23}$, $\beta_4=s_1s_2s_1\left(\alpha_{23}\right)=\alpha_{21}=-\alpha_{12}$. Rewriting the product as $s_2s_1s_2s_2$, we get the new order on the multiset $\alpha_{23}<\alpha_{13}<\alpha_{12}<\alpha_{21}$, and the last two reflections corresponding to the repeating root cancel, $s_2s_2=1$.
\end{example}

\subsection{Intertwining operators and a map of Verma modules}
In this section we construct a map of Verma modules $M_{\chi_D}\to M_{\w^{-1}\chi_D}$. The construction is motivated by the analogous result from \cite{GNP}, made slightly simpler by the fact we are dealing with semisimple modules, and slightly more complicated by the fact that we are dealing with degenerate double affine Hecke algebras. 

We first define several rational functions of complex variables $z=\left(z_i\right)_{i\in \ZZ}$ with values in the group algebra $\CC[\W]$. We allow infinitely many variables $z_i$, but only finitely many ever appear in any formula.
\begin{itemize}
\item Let $\varphi_i\left(z\right)=s_i+\frac{1}{z_i-z_{i+1}}$. If $V$ is an $\H$ module and $V[\xi]$ a $\left(\u\right)$ eigenspace such that $\xi_i\ne \xi_{i+1}$, then $\Phi_i|_{V[\xi]}=\varphi_i\left(\xi\right)$.
\item For a root $\alpha_{ij}$, let $\varphi_k^{ij}\left(z\right)=\varphi_k^{\alpha_{ij}}\left(z\right)=s_k+\frac{1}{z_i-z_{j}}$. In particular, $\varphi_i=\varphi_i^{\alpha_i}=\varphi_i^{i,i+1}$.
\item For $w\in \W$, let $^w\varphi_i\left(z\right)=\varphi_i\left(wz\right)=\varphi_i^{w^{-1}\alpha_i}\left(z\right)$.
\item For a product $s_{i_1}\ldots s_{i_l}$ (not necessarily reduced), let $$\varphi_{s_{i_1}\ldots s_{i_l}}\left(z\right)=\varphi_{i_1}\left(s_{i_2}\ldots s_{i_l} z\right)\ldots \varphi_{i_{l-1}}\left(s_{i_l}z\right) \varphi_{i_l}\left(z\right).$$
\item The following identities are easy to check: 
$$\textrm{if } s_is_j=s_js_i  \textrm{ then }  \varphi_{i}\left(s_jz\right)\varphi_j\left(z\right)=\varphi_{j}\left(s_iz\right)\varphi_i\left(z\right)$$
$$\varphi_{i}\left(s_{i+1}s_iz\right)\varphi_{i+1}\left(s_iz\right)\varphi_i\left(z\right)=\varphi_{i+1}\left(s_{i}s_{i+1}z\right)\varphi_{i}\left(s_{i+1}z\right)\varphi_{i+1}\left(z\right).$$
It follows that, for any reduced decomposition $w=s_{i_1}\ldots s_{i_l}$ of an element $w\in \W^0$, the function $\varphi_{s_{i_1}\ldots s_{i_l}}\left(z\right)$ depends only on $w$ and not on the choice of decomposition. Let $\varphi_w\left(z\right)=\varphi_{s_{i_1}\ldots s_{i_l}}\left(z\right)$.
\item Call a product $\varphi_{i_1}^{\beta_1}\ldots \varphi_{i_l}^{\beta_l}$ \emph{matched} if $\beta_j=\left(s_{i_j}\ldots s_{i_l}\right)^{-1}\left(\alpha_{i_j}\right)$. Matching products come from (possibly nonreduced) words in $\W^0$, $\left(\varphi_{i_1}^{\beta_1}\ldots \varphi_{i_l}^{\beta_l}\right)\left(z\right)=\varphi_{s_{i_1}\ldots s_{i_l}}\left(z\right)$. For example, $\varphi_1^{13}\varphi_2^{23}=\varphi_{s_1s_2}$ is matched, while $\varphi_1^{12}\varphi_2^{23}$ is not. 
\item If $\varphi_{i_1}^{\beta_1}\ldots \varphi_{i_l}^{\beta_l}$ is a matched product, then the two $l$-tuples $\left(i_1,\ldots, i_l\right)$ and $\left(\beta_1,\ldots , \beta_l\right)$ mutually determine each other. Sometimes we will use it to write   
$$\varphi_{i_1}^{\beta_1}\ldots \varphi_{i_l}^{\beta_l}=\varphi_{i_1}^{\bullet}\ldots \varphi_{i_l}^{\bullet}=\varphi^{\beta_1}\ldots \varphi^{\beta_l}.$$
\item The following identities hold if all the products are matched and $i,j,k,l$ are all distinct,
$$\varphi^{ij}\varphi^{ik}\varphi^{jk}=\varphi^{jk}\varphi^{ik}\varphi^{ij}$$
$$\varphi^{ij}\varphi^{kl}=\varphi^{kl}\varphi^{ij}.$$
\item If $\varphi_{s_{i_1}\ldots s_{i_k}}=\varphi_{i_1}^{\beta_1} \ldots \varphi_{i_k}^{\beta_k}$ is a matched product such that not all roots $\beta_j$ are distinct, then $s_{i_1}\ldots s_{i_k}$ is not a reduced word in $\W^0$. If $\beta_{j}=-\beta_{j'}$, $j<j'$, then the product $s_{i_1}\ldots s_{i_{j'-1}}$ can be rewritten as a product of simple reflections ending in $s_{j'}$, and the product $s_{i_{j+1}}\ldots s_{i_k}$ can be rewritten to start with $s_{i_j}$.
\item When this happens, two adjacent terms can be cancelled, producing a scalar function of $z$: 
$$\varphi_k^{ji}\left(z\right)\varphi_k^{ij}\left(z\right)=\left(s_k-\frac{1}{z_i-z_{j}}\right)\left(s_k+\frac{1}{z_i-z_{j}}\right)=1-\frac{1}{\left(z_i-z_{j}\right)^2}.$$
\end{itemize}

Another lemma gives us a shorter expression for the functions we will need later:
\begin{lemma}\label{varphiwasaproduct}
For a reduced expression $w=s_{i_1}\ldots s_{i_l}\in \W^0$, and $\beta_k=s_{i_1}\ldots s_{i_{k-1}}\left(\alpha_{i_k}\right)\in \mathcal{I}_w$ as above, 
$$\varphi_w\left(w^{-1}z\right)= \left( \varphi_{i_1}^{\beta_1}\varphi_{i_2}^{\beta_2} \ldots \varphi_{i_l}^{\beta_l}\right) \left(-z\right)=\left(\varphi_{i_1}^{-\beta_1}\varphi_{i_2}^{-\beta_2} \ldots \varphi_{i_l}^{-\beta_l}\right) \left(z\right) .$$
\end{lemma}
\begin{proof} As $w^{-1}=s_{i_l}\ldots s_{i_1}$, we have
\begin{align*}
\varphi_w\left(w^{-1}z\right) &= \varphi_{i_1}\left(s_{i_2}\ldots s_{i_l}w^{-1} z\right)\ldots \varphi_{i_{l-1}}\left(s_{i_l}w^{-1}z\right) \varphi_{i_l}\left(w^{-1}z\right) \\
&=  \varphi_{i_1}\left(s_{i_1}z\right)\ldots \varphi_{i_{l-1}}\left(s_{i_{l-1}}\ldots s_{i_1} z\right) \varphi_{i_l}\left(s_{i_l}\ldots s_{i_1}z\right) \\
&= \varphi_{i_1}^{s_{i_1}\left(\alpha_{i_1}\right)}\left(z\right)\ldots \varphi_{i_{l-1}}^{s_{i_1}\ldots s_{i_{l-1}}\left(\alpha_{i_{l-1}}\right)}\left(z\right) \varphi_{i_l}^{s_{i_1}\ldots s_{i_{l}}\left(\alpha_{i_{l}}\right)}\left(z\right)\\
&= \varphi_{i_1}^{-\beta_1}\left(z\right)\ldots \varphi_{i_{l}}^{-\beta_l}\left(z\right) \\
&= \varphi_{i_1}^{\beta_1}\left(-z\right)\ldots \varphi_{i_{l}}^{\beta_l}\left(-z\right). \qedhere
\end{align*}
\end{proof}

For a periodic skew diagram $D$, we want to evaluate the function $\varphi_{\w}\left(\w^{-1}z\right)$ at the point $z=\chi_D$. The obstacle is that some factors might have poles. More specifically, $\varphi_{k}^{ij}\left(z\right)$ has a pole along the divisor $z_{i}=z_j$, and for $z=\chi_D$, $z_i=z_j$ if and only if $i$ and $j$ are on the same diagonal of the row reading tableau $T_0$ on $D$. In that case, $\alpha_{i,j-1}$ and $\alpha_{ij}$ are adjacent in the special convex order by Lemma \ref{special}, and by Lemma \ref{convex}, we can reorder the roots after $\alpha_{ij}$ to get another convex order in which $\alpha_{i,j-1}<\alpha_{ij}<\alpha_{j-1,j}$ are adjacent.

\begin{lemma}\label{31}
Let $w\in \W^0$, and assume that $\alpha_{i,j-1}<\alpha_{ij}<\alpha_{j-1,j}$ are adjacent in some convex order on $\mathcal{I}_w$. Then, in the reduced decomposition of $w$ into simple reflections corresponding to this convex order, the factors corresponding to these roots are $s_{a}s_{a+1}s_a$ for some $a$. 
\end{lemma}
\begin{proof}
Let $\alpha_{i,j-1}$ be the $k$-th root in that particular convex order on $\mathcal{I}_w$, so that $\beta_k=\alpha_{i,j-1}$, $\beta_{k+1}=\alpha_{i,j}$ and $\beta_{k+2}=\alpha_{j-1,j}$. Let $s_a,s_b,s_c$ be the simple reflections corresponding to those factors in the reduced decomposition of $w$ corresponding to this convex order, so $w=w's_as_bs_cw''$ for some $w',w''\in \W^0$.

By the rule for associating roots $\beta_k,\beta_{k+1},\beta_{k+2}$ to the reduced decomposition of $w$, 
\begin{align*}
\alpha_{i,j-1}&=w'\left(\alpha_a\right)\\ 
\alpha_{i,j}&=w's_a\left(\alpha_b\right)\\
\alpha_{j-1,j}&=w's_as_b\left(\alpha_c\right).
\end{align*}
The first of these implies that $i=w'\left(a\right)$ and $j-1=w'\left(a+1\right)$. From the second of these equalities it follows that $i=w's_a\left(b\right)$, which can be rewritten as $b=s_a w'^{-1}(i)=s_a\left(a\right)=a+1$. The third one implies that $j-1=w's_as_b\left(c\right)$, so $c=s_bs_aw'^{-1}\left(j-1\right)=s_{a+1}s_a\left(a+1\right)=a$.  
\end{proof}

So, if $\alpha_{i,j-1}<\alpha_{ij}<\alpha_{j-1,j}$ are adjacent in some convex order on $\mathcal{I}_w$, then the part of the product $\varphi_w\left(w^{-1}z\right)$ corresponding to them is the factor $\left(\varphi_a^{-\alpha_{i,j-1}}\varphi_{a+1}^{-\alpha_{i,j}}\varphi_a^{-\alpha_{j-1,j}}\right)\left(z\right)$.

\begin{lemma}\label{32}
Assume that $z_{j-1}=z_j-1$. Then $$\varphi_a^{j-1,i}\left(z\right)\varphi_{a+1}^{j,i}\left(z\right)\varphi_a^{j,j-1}\left(z\right)=\left(s_a s_{a+1}-\frac{1}{z_j-z_{j-1}}s_{a+1}-\frac{1}{z_j-z_{j-1}} \right) \cdot \varphi_a^{j,j-1}\left(z\right).$$
In particular, it is regular along the divisor $z_i=z_j$.
\end{lemma}
\begin{proof}
Let $z_i-z_j=\varepsilon\in \CC$. Then $z_i-z_{j-1}=\varepsilon+1$. We calculate:
\begin{align*}
\varphi_a^{j-1,i}\left(z\right)\varphi_{a+1}^{j,i}\left(z\right)\varphi_a^{j,j-1}\left(z\right)&=\left(s_a-\frac{1}{z_i-z_{j-1}}\right)\left(s_{a+1}-\frac{1}{z_i-z_{j}}\right)\left(s_a-\frac{1}{z_{j-1}-z_{j}}\right)\\
&=\left(s_a-\frac{1}{\varepsilon+1}\right)\left(s_{a+1}-\frac{1}{\varepsilon}\right)\left(s_a+1\right)\\
&=\left(s_a-\frac{1}{\varepsilon+1}\right)s_{a+1}\left(s_a+1\right)-\frac{1}{\varepsilon}\left(s_a-\frac{1}{\varepsilon+1}\right)\left(s_a+1\right)\\
&=\left(s_a-\frac{1}{\varepsilon+1}\right)s_{a+1}\left(s_a+1\right)-\frac{1}{\varepsilon}\cdot \frac{\varepsilon}{\varepsilon+1}\left(s_a+1\right)\\
&=\left(s_a s_{a+1}-\frac{1}{\varepsilon+1}s_{a+1}-\frac{1}{\varepsilon+1} \right) \cdot \left(s_a+1\right),
\end{align*}
which proves the first claim. In particular, its limit at $z_i=z_j$ is 
\begin{align*}
\lim_{\varepsilon\to 0}\left(s_a s_{a+1}-\frac{1}{\varepsilon+1}s_{a+1}-\frac{1}{\varepsilon+1} \right) \cdot \left(s_a+1\right)&=\left(s_a s_{a+1}-s_{a+1}-1 \right) \cdot \left(s_a+1\right).\qedhere
\end{align*}
\end{proof}

We are now ready to prove a key proposition. 

\begin{proposition}\label{regular}
For a periodic skew diagram $D$ and the permutation $\w$ from Lemma \ref{defw}, the function
$\varphi_{\w}\left(\w ^{-1}z\right)$, restricted to the set 
$$\mathcal{F}_D=\{ \left(z_i\right)_{i} | \forall \,\, i,j \textrm{ in the same row of } T_0 \,\,, z_i-z_j=\chi_i-\chi_j \}$$
is regular and nonzero in a neighbourhood of the point $z=\chi_D\in \mathcal{F}_D$.
\end{proposition}
We allow arbitrarily many $z_i$, but $\varphi_{\w}\left(\w ^{-1}z\right)$ only depends on finitely many of them, namely on those $z_i$ such that there exists $j\in \mathbb{Z}$ such that $\pm \alpha_{i,j} \in \mathcal{I}_{\w}$. The statement should be interpreted in that way, i.e. $\mathcal{F}_D \subset \CC^N$ for some $N$. 
\begin{proof}
Let $\w=s_{i_1}\ldots s_{i_l}$ be the reduced decomposition corresponding to the special order from Section \ref{specialorder}. By Lemma \ref{varphiwasaproduct},  
$$\varphi_{\w}\left(\w ^{-1}z\right)=\left(\varphi_{i_1}^{-\beta_1}\varphi_{i_2}^{-\beta_2}\ldots \varphi_{i_l}^{-\beta_l}\right)\left(z\right).$$
The factors of this product which have a pole at $z=\chi_D$ are those for which $\beta_k=\alpha_{ij}$ with $\chi_i=\chi_j$, meaning that $i$ and $j$ are on the same diagonal of the row reading tableau $T_0$ on $D$. By Lemma \ref{special}, in that case the affine root preceding $\alpha_{ij}$ in the special order is $\beta_{k-1}=\alpha_{i,j-1}$, which does not have a pole. Restricted to $\mathcal{F}_D$ and in the neighbourhood of $z=\chi_D$, all other terms can be evaluated. To evaluate $\varphi_{\w}\left(\w ^{-1}z\right)|_{\mathcal{F}_D}$ at $z=\chi_D$, we proceed as follows, evaluating terms from left to right:
\begin{itemize}
\item If $k$ is such that neither the $k-$th term $\varphi_{i_k}^{-\beta_k}$ nor the $\left(k+1\right)-$st term $\varphi_{i_{k+1}}^{-\beta_{k+1}}$ have a pole, then evaluate it at $\chi_D$, getting $\varphi_{i_k}^{-\beta_k}\left(\chi_D\right)$.
\item If terms up to $\left(k-1\right)$-st have been evaluated, the $k$-th term needs to be evaluated next and the $\left(k+1\right)$-st term has a pole, then we evaluate the $k$-th and $\left(k+1\right)$-st term together. 
 By Lemma \ref{special}, $\beta_k=\alpha_{i,j-1}$, $\beta_{k+1}=\alpha_{i,j}$, and $i$ and $j$ are on the same diagonal in $D$. Using Lemma \ref{convex}, we can reorder terms $k+2,\ldots, l$, keeping terms $1,\ldots, k+1$ fixed, so that the new order is again convex and represents a different reduced decomposition of $\w$, and so that in the new order the $\left(k+2\right)$-nd term is $\beta_{k+2}=\alpha_{j-1,j}$. By Lemma \ref{31}, the $k$-th, $\left(k+1\right)$-st and $\left(k+2\right)$-nd terms are then $\left(\varphi_{a}^{j-1,i}\varphi_{a+1}^{j,i}\varphi_{a}^{j,j-1}\right)\left(z\right)$. By Lemma \ref{32}, the limit of this product along $\mathcal{F}_D$ at $\chi_D$ is $\left(s_a s_{a+1}-s_{a+1}-1\right) \cdot \varphi_a^{j,j-1}\left(z\right)|_{z=\chi_D}$. Now reorder the terms $k+2,\ldots ,  l$ again back to their original order, bringing $\varphi_a^{-\alpha_{j-1,j}}\left(z\right)$ back to its original place. The net effect of this step was to replace the product of the $k$-th and $\left(k+1\right)$-st term by $\left(s_a s_{a+1}-s_{a+1}-1\right)$, keeping all the other terms fixed. 
\end{itemize}

In this way, one can evaluate all terms of $\varphi_{\w}\left(\w ^{-1}z\right)|_{\mathcal{F}_D}$ at $z=\chi_D$, so this rational function is regular at that point. To see that it is nonzero, notice that $\varphi_{\w}\left(\w ^{-1}\chi_D\right)=\w+$ a linear combination of shorter terms, so in particular it is nonzero in $\CC[\W]$. 
\end{proof}

The procedure described here is called \emph{fusion}. In later computations, we will call the terms $\varphi_{a}^{-\alpha_{i,j-1}}\varphi_{a+1}^{-\alpha_{i,j}}$ fused, and refer to $\varphi_{a}^{-\alpha_{i,j-1}}$, $\varphi_{a+1}^{-\alpha_{i,j}}$ and $\varphi_{a}^{-\alpha_{i,j-1}}$ as the first, second and third term of a fusion.

Another key proposition is
\begin{proposition}\label{mapF}
The vector $$E_D=\varphi_{\w}\left(\w ^{-1}z\right)|_{z=\chi_D\in \mathcal{F}_D}\One_{\w^{-1}\chi_D}\in M_{\w^{-1}\chi_D}$$
is an eigenvector for $\u$ with the eigenvalue $\chi_D$. It determines a nonzero $\H$ homomorphism of Verma modules $F:M_{\chi_D} \to M_{\w^{-1}\chi_D}$ by $F\left(\One_{\chi_D}\right)=\One_{\w^{-1}\chi_D}$.
\end{proposition}
\begin{proof}
The vector $\varphi_{\w}\left(\w ^{-1}z\right)|_{z=\chi_D\in \mathcal{F}_D}\One_{\w^{-1}\chi_D}\in M_{\w^{-1}\chi_D}$
is well defined and nonzero by the previous lemma. The only thing to prove is that it is an eigenvector $\u$ with the eigenvalue $\chi_D$.

Define $$v_k=\left( \varphi_{i_{k+1}}^{-\beta_{k+1}}\ldots \varphi_{i_{l}}^{-\beta_l}\right) |_{z=\chi_D\in \mathcal{F}_D} \One_{\w^{-1}\chi_D}$$
for all $0\le k\le l$ such that $\varphi_{i_{k+1}}^{-\beta_{k+1}}\left(z\right)$ does not have a pole at $z=\chi_D$. This is a regular and nonzero in the limit $z\to \chi_D$ along $z\in \mathcal{F}_D$. 

We will prove by downwards induction on $k$, starting from $k=l$ and going down by steps of $1$ or $2$, that $v_k$ is an eigenvector for $\u$ with the eigenvalue $s_{i_{k+1}}\ldots s_{i_l} \w^{-1}\chi_D$.

For $k=l$, $v_l=\One_{\w^{-1}\chi_D}\in M_{\w^{-1}\chi_D}$, and this is an eigenvector with the eigenvalue $\w^{-1}\chi_D$. Assume we have proved the claim for $v_k$. In the step if induction, we will distinguish two cases, and prove the corresponding claim for either $v_{k-1}$ or $v_{k-2}$.

{\bf First case:} Assume that $\varphi_{i_{k}}^{-\beta_{k}}\left(z\right)$ does not have a pole at $z=\chi_D$. Then the intertwiner $\Phi_{i_k}$ is well defined (it does not have a pole) on the eigenspace $s_{i_{k+1}}\ldots s_{i_l} \w^{-1}\chi_D$, and, restricted to it, $\Phi_{i_k}=\varphi_{i_k}^{-\beta_k}\left(\chi_D\right)$. Thus, $v_{k-1}=\Phi_{i_k}v_k$, which is an eigenvector for $\u$ with the eigenvalue $s_{i_k}s_{i_{k+1}}\ldots s_{i_l} \w^{-1}\chi_D$.

{\bf Second case:} Assume that $\varphi_{i_{k}}^{-\beta_{k}}\left(z\right)$ has a pole at $z=\chi_D$. We will prove the claim for $v_{k-2}=\varphi_{i_{k-1}}^{-\beta_{k-1}}\varphi_{i_{k}}^{-\beta_{k}}v_{k}$. 

By the definition of the special order, and by Lemma \ref{31}, there exists some integers $a,i,j$ such that $\varphi_{i_{k-1}}^{-\beta_{k-1}}\varphi_{i_{k}}^{-\beta_{k}}=\varphi_{a}^{-\alpha_{i,j-1}}\varphi_{a+1}^{-\alpha_{i,j}}$. By the fusion procedure from Lemma \ref{regular}, in the limit $z\to \chi_D$ along $\mathcal{F}_D$, the product $\varphi_{i_{k-1}}^{-\beta_{k-1}}\left(z\right)\varphi_{i_k}^{-\beta_k}\left(z\right)$ is replaced by the fused factor $\left(s_{a}s_{a+1}-s_{a+1}-1\right)$.

\begin{comment}
we see that: 
$$\varphi_{a+1}^{-\alpha_{ij}}\left(z\right)=\varphi_{i_k}^{-\beta_k}\left(z\right)=\varphi_{i_k}\left(s_{i_{k+1}}\ldots s_{i_l} \w^{-1}z\right)=\varphi_{a+1}^{\alpha_{a+1}}\left(s_{i_{k+1}}\ldots s_{i_l} \w^{-1}z\right)$$
$$\varphi_{a}^{-\alpha_{i,j-1}}\left(z\right)=\varphi_{i_{k-1}}^{-\beta_{k-1}}\left(z\right)=\varphi_{i_{k-1}}\left(s_{i_k}s_{i_{k+1}}\ldots s_{i_l} \w^{-1}z\right)=\varphi_{a}^{\alpha_{a}}\left(s_{i_k}s_{i_{k+1}}\ldots s_{i_l} \w^{-1}z\right).$$
From this it follows that
\begin{align*}
z_j&=\left(s_{i_{k+1}}\ldots s_{i_l} \w^{-1}z\right)_{a+1}\\
z_i&=\left(s_{i_{k+1}}\ldots s_{i_l} \w^{-1}z\right)_{a+2}\\
z_{j-1}&=\left(s_{i_k}s_{i_{k+1}}\ldots s_{i_l} \w^{-1}z\right)_{a}\\
z_i&=\left(s_{i_k}s_{i_{k+1}}\ldots s_{i_l} \w^{-1}z\right)_{a+1}.
\end{align*}
Set $z=\chi_D$ in the above equations; then for $c=\left(\chi_D\right)_i$ we have $z_i=z_j=c$, $z_{j-1}=c-1$. Use also that $i_k=a+1$. Let $\nu=s_{i_{k+1}}\ldots s_{i_l} \w^{-1}\chi_D$ be the eigenvalue of $v_k$. Then the above equations become:
\begin{align*}
\nu_{a+1}&=c\\
\nu_{a+2}&=c\\
\nu_{a}&=c-1.
\end{align*}
\end{comment}

Using that the product defining $v_{k-2}$ is matched and that $\varphi_{a+1}^{j,i}$ has a at $z=\chi_D$, we see that the eigenvaleus of $u_{a},u_{a+1},u_{a+2}$ on $v_k$ are, respectively, $c-1,c,c$ for $c=\left(\chi_D\right)_i$. Additionally, we can reorder  $s_{i_{k+1}}\ldots s_{i_{l}}$ so that $\varphi_{i_{k-1}}^{-\beta_{k-1}}=\varphi_a^{-\alpha_{j-1,j}}=\left(s_a+1\right)$. In particular, $\left(s_a-1\right)v_k=0.$

We now calculate the action of $u_i$ on $v_{k-2}=\left(s_{a}s_{a+1}-s_{a+1}-1\right)v_k$ for all $i=1,\ldots ,n$.
\begin{itemize}
\item If $i\ne a,a+1, a+2$, then $u_i\left(s_{a}s_{a+1}-s_{a+1}-1\right)=\left(s_{a}s_{a+1}-s_{a+1}-1\right)u_i$ and $\left(s_{a}s_{a+1}\nu\right)_i=\nu_i$, so $v_{k-2}$ is an eigenvector for $u_i$ with the eigenvalue $\left(s_{i_{k-1}}s_{i_{k}}s_{i_{k+1}}\ldots s_{i_l} \w^{-1}\chi_D\right)\left(u_i\right)$. 
\item If $i=a$, then 
\begin{align*}
u_iv_{k-2}&=u_a\left(s_{a}s_{a+1}-s_{a+1}-1\right)v_k\\
&=\left(s_as_{a+1}u_{a+2}-s_a-s_{a+1}-s_{a+1}u_{a}-u_a\right)v_k\\
&=\left(s_as_{a+1}c-s_a-s_{a+1}-s_{a+1}\left(c-1\right)-\left(c-1\right)\right)v_k\\
&=c\left(s_{a}s_{a+1}-s_{a+1}-1\right)v_k-\left(s_a-1\right)v_k\\
&=cv_{k-2}.
\end{align*}
\item If $i=a+1$, then 
\begin{align*}
u_iv_{k-2}&=u_{a+1}\left(s_{a}s_{a+1}-s_{a+1}-1\right)v_k\\
&=\left(s_{a}s_{a+1}u_{a}+s_{a+1}-s_{a+1}u_{a+2}+1-u_{a+1}\right)v_k\\
&=\left(s_{a}s_{a+1}-s_{a+1}-1\right)\left(c-1\right)v_k\\
&= \left(c-1\right)v_{k-2}.
\end{align*}
\item If $i=a+2$, then 
\begin{align*}
u_iv_{k-2}&=u_{a+2}\left(s_{a}s_{a+1}-s_{a+1}-1\right)v_k\\
&=\left(s_{a}s_{a+1}u_{a+1}+s_{a}-s_{a+1}u_{a+1}-1-u_{a+2}\right)v_k\\
&=\left(\left(s_{a}s_{a+1}-s_{a+1}-1\right)c+\left(s_a-1\right)\right)v_k\\
&= cv_{k-2}.
\end{align*}
\end{itemize}

So, $v_{k-2}$ is an eigenvector for $\u$ with the eigenvalue $s_{i_{k-1}}\ldots s_{i_l} \w^{-1}\chi_D$. This finishes the induction argument. 

In particular, the vector $v_0=\varphi_{\w}\left(\w ^{-1}z\right)|_{z=\chi_D\in \mathcal{F}_D}\One_{\w^{-1}\chi_D}$ is an eigenvector for $\u$ with the eigenvalue $s_{i_1}\ldots s_{i_l} \w^{-1}\chi_D=\w \w^{-1}\chi_D=\chi_D$.

Verma modules are induced modules, so the eigenvector $E_D\in M_{\w^{-1}\chi_D}$ induces a homomorphism of Verma modules $F:M_D\to M_{\w^{-1}\chi_D}$. Because $\varphi_{\w}\left(\w ^{-1}z\right)|_{z=\chi_D\in \mathcal{F}_D}\in \CC[\W]$ is nonzero and Verma module $M_{\w^{-1}\chi_D}$ is free as a $\CC[\W]$ module, it follows that $E_D$ is nonzero, and the morphism $F$ is nonzero as well. 
\end{proof}

\begin{example}\label{example4e}
For $D$ and $\w$ as in Examples \ref{example4}, \ref{example4b}, \ref{example4d},
\begin{align*}
E_D&=\varphi_{\w}\left(\w ^{-1}z\right)|_{z=\chi_D\in \mathcal{F}_D}\One_{\w^{-1}\chi_D}\\
&=\lim_{\substack{z=\left(\varepsilon,1+\varepsilon, -1,0\right)\\ \epsilon\to 0}} \left(\varphi_{1}^{12} \varphi_{2}^{13}\varphi_{3}^{14}\varphi_{2}^{34}\varphi_{0}^{-3,2}\varphi_{1}^{-3,4}\varphi_{3}^{-1,2}\varphi_{0}^{-1,4}\varphi_{3}^{24}\right)\left(-z\right) \mathbf{1}_{\left(2,3,-3,-2\right)}\\
&=\lim_{\varepsilon\to 0}\left(\varphi_{1}^{21} \varphi_{2}^{31}\varphi_{3}^{41}\varphi_{2}^{43}\varphi_{0}^{2,-3}\varphi_{1}^{4,-3}\varphi_{3}^{2,-1}\varphi_{0}^{4,-1}\varphi_{3}^{42}\right)\left(z\right) \mathbf{1}_{\left(2,3,-3,-2\right)}\\
&=\lim_{\varepsilon\to 0}\left(s_1-\frac{1}{z_1-z_2}\right)\left(s_2-\frac{1}{z_1-z_3}\right) \left(s_3-\frac{1}{z_1-z_4}\right)\left(s_2-\frac{1}{z_3-z_4}\right)\cdot\\
&\cdot\left(s_0-\frac{1}{z_{-3}-z_2}\right)\left(s_1-\frac{1}{z_{-3}-z_4}\right)\left(s_3-\frac{1}{z_{-1}-z_2}\right)\left(s_0-\frac{1}{z_{-1}-z_4}\right)\left(s_3-\frac{1}{z_2-z_4}\right)\mathbf{1}_{\left(2,3,-3,-2\right)}\\
&=\lim_{\varepsilon\to 0}\left(\left(s_1+1\right)\left(s_2-\frac{1}{1+\varepsilon}\right) \left(s_3-\frac{1}{\varepsilon}\right)\left(s_2+1\right)\right.\\
&\left. \left(s_0-\frac{1}{2}\right)\left(s_1-\frac{1}{3+\varepsilon}\right)\left(s_3-\frac{1}{1-\varepsilon}\right)\left(s_0-\frac{1}{2}\right)\left(s_3-\frac{1}{1+\varepsilon}\right)\right) \mathbf{1}_{\left(2,3,-3,-2\right)}\\\\
&=\left(s_1+1\right)\left(s_2s_3-s_3-1\right)\left(s_2+1\right)\left(s_0-\frac{1}{2}\right)\left(s_1-\frac{1}{3}\right)\left(s_3-1\right)\left(s_0-\frac{1}{2}\right)\left(s_3-1\right)\mathbf{1}_{\left(2,3,-3,-2\right)}.
\end{align*}
\end{example}

\subsection{The map $F:M_{\chi_D} \to M_{\w^{-1}\chi_D}$ factors through $L_D$}
To prove that the homomorphism $F$ constructed in the last section as a map between Verma modules induces an inclusion of the irreducible module $L_D$ into the Verma module $M_{\w^{-1}\chi_D}$, we have to show that it is zero on the kernel of the quotient map $Q:M_{\chi_D}\to L_D$. By Theorem \ref{defss}, the kernel of $Q$ is generated as an $\H$-module by 
$$\{ \Phi_i \Phi_w \mathbf{1}_D | wT_0 \textrm{ standard, } s_iwT_0 \textrm{ not standard}\}.$$  Let us study this set first, before proving that $F|_{\Ker Q}=0$.

\begin{lemma}\label{shorter}
Assume that $D$ is a periodic skew diagram and $w\in \W$ such that $wT_0$ is standard and $s_iwT_0$ is not standard. Then $l\left(s_iw\w\right)<l\left(w\w\right)$.
\end{lemma}
\begin{proof}
By \cite{H}, Lemma 1.6, $l\left(s_iw\w\right)=l\left(w\w\right)-1$ if and only if $\alpha_{i,i+1}$ is an inversion of $\left(w\w\right)^{-1}$. The permutation $w\w$ of $\ZZ$ is obtained by reading the entries of $wT_0$, up each column to its end, and continuing up the next column to the left. By an analogue of Lemma \ref{Iw}, the set $\mathcal{I}_{w\w}$ of inversions of $\left(w\w\right)^{-1}$ is the set of all roots $\alpha_{jk}$ such that the arrow $\overrightarrow{jk}$ on the tableau $wT_0$ is pointing right, pointing down, or pointing down and right. By Lemma \ref{adjacent}, the boxes containing $i$ and $i+1$ in $wT_0$ are adjacent, and the arrow $\overrightarrow{i,i+1}$ is pointing one box right or pointing one box down, so $\alpha_{i,i+1}$ is an inversion of $\left(w\w\right)^{-1}$ and $l\left(s_iw\w\right)=l\left(w\w\right)-1$.
\end{proof}

\begin{lemma}\label{row}
Assume that $D$ is a periodic skew diagram with no infinite column, $\kappa \ge 2$, and $F:M_D\to M_{\w^{-1}D}$ as defined in Proposition \ref{mapF}. For every $w\in \W$ such that $wT_0$ is standard, $s_{i}wT_0$ is not, and $i$ and $i+1$ are in the same row of $wT_0$, we have $$F\left(\Phi_i\Phi_w\mathbf{1}_D\right)=0.$$
\end{lemma}
\begin{proof}
Using that $\varphi_{s_i w}\left(z\right)$ is regular at $z= \chi_D$, we get
\begin{align*}
F\left(\Phi_i\Phi_w \mathbf{1}_D\right)&= \Phi_i\Phi_w F\left(\mathbf{1}_D\right)\\
&= \varphi_{s_iw}\left(\chi_D\right) \lim_{\substack{z\to \chi_D \\ z\in \mathcal{F}_D}} \varphi_{\w}\left(\w^{-1}\left(z\right)\right)\mathbf{1}_{\w^{-1}\chi_D}\\
&= \lim_{\substack{z\to \chi_D \\ z\in \mathcal{F}_D}}  \varphi_{s_i}\left(w\left(z\right)\right)\varphi_{w\w}\left(\w^{-1}\left(z\right)\right)\mathbf{1}_{\w^{-1}\chi_D}.
\end{align*}

By Lemma \ref{shorter}, $l\left(s_iw\w\right)<l\left(w\w\right)$, so it is possible to write $w\w=s_iw_1$, with $l\left(w_1\right)=l\left(w\w\right)-1$. Use this decomposition to write $\varphi_{w\w}\left(\w^{-1}\left(z\right)\right)$, and get
\begin{align*}
F\left(\Phi_i\Phi_w \mathbf{1}_D\right)&= \lim_{\substack{z\to \chi_D \\ z\in \mathcal{F}_D}}  \varphi_{s_i}\left(w\left(z\right)\right) \varphi_{s_i}\left(s_iw\left(z\right)\right) \varphi_{w_1}\left(\w^{-1}\left(z\right)\right)\mathbf{1}_{\w^{-1}\chi_D} \\
&= \lim_{\substack{z\to \chi_D \\ z\in \mathcal{F}_D}} \left(s_i-1\right)\left(s_i+1\right) \varphi_{w_1}\left(\w^{-1}\left(z\right)\right)\mathbf{1}_{\w^{-1}\chi_D}\\
&= \lim_{\substack{z\to \chi_D \\ z\in \mathcal{F}_D}} 0\cdot \varphi_{w_1}\left(\w^{-1}\left(z\right)\right)\mathbf{1}_{\w^{-1}\chi_D}\\
&= 0. \qedhere
\end{align*}
\end{proof}

\begin{example}\label{example4f}

Let $D$ be the periodic diagram from examples \ref{example4}, \ref{example4b}, \ref{example4d} and \ref{example4e}. Let $w=1$, and $s_i=s_1$. Then $s_iwT_0=s_1T_0$ is not standard, as $1$ and $2$ are in adjacent boxes of $T_0$ in the first row, $1=T_0\left(1,1\right)$, $2=T_0\left(1,2\right)$. The vector $E_D$ can be written as $E_D=\left(s+1\right)E'$, where $E'=\left(s_2s_3-s_3-1\right)\left(s_2+1\right)\left(s_0-\frac{1}{2}\right)\left(s_1-\frac{1}{3}\right)\left(s_3-1\right)\left(s_0-\frac{1}{2}\right)\left(s_3-1\right)\mathbf{1}_{\w^{-1}\chi_D}\in M_{\w^{-1}\chi_D}$, so 
$$F\left(\Phi_{1}\mathbf{1}_{\chi_D}\right)=\Phi_{1}E_D=\left(s_1-1\right)\left(s_1+1\right)E'=0.$$
Writing the function corresponding to the non-reduced element $s_iw\w=s_1\w$,  
$$\varphi_{s_1\w}\left(\w^{-1}z\right)=\left( \varphi_{1}^{12}\varphi_{1}^{21} \varphi_{2}^{31}\varphi_{3}^{41}\varphi_{2}^{43}\varphi_{0}^{2,-3}\varphi_{1}^{4,-3}\varphi_{3}^{2,-1}\varphi_{0}^{4,-1}\varphi_{3}^{42}\right)\left(z\right);$$ we see that the root $\alpha_{12}$ repeats, and the corresponding terms cancel to give $\varphi_{1}^{12}\varphi_{1}^{21}=0$ at $z=\chi_D$.
 \end{example}

When $i$ and $i+1$ are adjacent and in the same column of $wT_0$, the proof that $F\left(\Phi_i\Phi_w\One_D\right)=0$ is more complicated. We first give two examples, which illustrate the two possible cases.

\begin{example}\label{example4g}
Let $D$ and $\w$ be the same as in example \ref{example4e}, and let $w=s_2$. The tableau $s_2T_0$ is standard, and contains $\raisebox{-6pt}{\young(13,24)}$,
so $s_3s_2T_0$ is not standard. The vector $\Phi_3 \Phi_2 \mathbf{1}$ is in the kernel of $Q$; let us show that it is also in the kernel of $F$.

By definition,
$$F\left(\Phi_3 \Phi_2 \mathbf{1}\right)=\lim_{\substack{z=\left(\varepsilon,1+\varepsilon, -1,0\right)\\ \epsilon\to 0}} \left( \varphi_{3}^{24}\varphi_{2}^{23}\varphi_{1}^{21}\varphi_{2}^{31}\varphi_{3}^{41}\varphi_{2}^{43}\varphi_{0}^{2,-3}
\varphi_{1}^{4,-3}\varphi_{3}^{2,-1}\varphi_{0}^{4,-1}\varphi_{3}^{42}\right)\left(z\right)\cdot  \mathbf{1}.$$
The fact that the product $s_3s_2\w$ is not reduced is reflected in the fact that the root $\alpha_{24}$ is repeated. The idea is to cancel the functions corresponding to the repeated root; for that, let us first rewrite the product so that the factors corresponding to the repeated root are adjacent. Let $w_1=s_2s_1s_2s_3s_2s_0s_1s_3s_0$; then $s_3s_2\w=s_3w_1s_3$. Direct computation shows that $s_3w_1=w_1s_3$, so $s_3s_2\w=w_1s_3s_3$. The corresponding function is
$$F\left(\Phi_3 \Phi_2 \mathbf{1}\right)=\lim_{\varepsilon\to 0} \left( \varphi_{2}^{43}\varphi_{1}^{41}\varphi_{2}^{31}\varphi_{3}^{21}\varphi_{2}^{23}\varphi_{0}^{4,-3}
\varphi_{1}^{2,-3}\varphi_{3}^{4,-1}\varphi_{0}^{2,-1}  \varphi_{3}^{24} \varphi_{3}^{42}\right)\left(z\right) \cdot  \mathbf{1}.$$
Notice that the lower indices in the reduced decomposition of $\varphi_{w_1}$ have stayed the same, and the upper have been changed by the transposition $s_{24}$. Now let us cancel the last two terms, giving $$\left(\varphi_{3}^{24} \varphi_{3}^{42}\right)\left(z\right)=\left(s_3+\frac{1}{z_2-z_4}\right)\left(s_3+\frac{1}{z_4-z_2}\right)=
\left(s_3-\frac{1}{1+\varepsilon}\right)\left(s_3+\frac{1}{1+\varepsilon}\right)=\frac{\left(2+\varepsilon\right)\varepsilon}{\left(1+\varepsilon\right)^2}.$$ 
This tends to zero at $\varepsilon\to 0$, but it is not identically equal to it. To show that the product tends to zero, we need to show that the remaining factor is regular at $\varepsilon\to 0$. This factor is 
$$\left( \varphi_{2}^{43}\varphi_{1}^{41}\varphi_{2}^{31}\varphi_{3}^{21}\varphi_{2}^{23}\varphi_{0}^{4,-3} \varphi_{1}^{2,-3}\varphi_{3}^{4,-1}\varphi_{0}^{2,-1}\right) \left(z\right).$$

All factors of this are regular at $z=\chi_D$ except $\varphi_{1}^{41}$. However, its product with the two neighbouring factors is again a fusion: 
$$\varphi_{2}^{43}\varphi_{1}^{41}\varphi_{2}^{31}\left(z\right)=\left(s_2+1\right)\left(s_1-\frac{1}{\varepsilon}\right)\left(s_2-\frac{1}{1+\varepsilon}\right)=\left(s_2+1\right)\left(s_1s_2-\frac{1}{1+\varepsilon}s_1-\frac{1}{1+\varepsilon}\right),$$
and this is regular at $\varepsilon=0$. 
\end{example}

When $i$ and $i+1$ are in the same column of $wT_0$, the proof that $\Phi_i\Phi_w\One_D$ is in $\Ker F$ will go along the same lines if $i$ and $i+1$ are in the rightmost two boxes of their rows. We will show in the next lemma that in that case, $\varphi_{s_iw\w}$ can be rewritten so that two terms cancel, their product tends to $0$, and the remaining factor is regular at $z=\chi_D$. In case $i$ and $i+1$ are not in the rightmost two boxes of their rows, additional cancelations are required to show the product tends to $0$. This is illustrated in the following example.

\begin{example}\label{example4h}
Let $D$ and $w=s_2$ be the same as in example \ref{example4g}, and let $s_i=s_1$. We have 
$$F\left(\Phi_1 \Phi_2 \mathbf{1}\right)=\lim_{\substack{z=\left(\varepsilon,1+\varepsilon, -1,0\right)\\ \epsilon\to 0}} \left( \varphi_{1}^{13}
\varphi_{2}^{23}\varphi_{1}^{21}\varphi_{2}^{31}\varphi_{3}^{41}\varphi_{2}^{43}\varphi_{0}^{2,-3}
\varphi_{1}^{4,-3}\varphi_{3}^{2,-1}\varphi_{0}^{4,-1}\varphi_{3}^{42}\right)\left(z\right)\cdot  \mathbf{1}.$$
The product $s_1w\w$ is not reduced and the root $\alpha_{13}$ is repeated. Let $w_1=s_2s_1$; then $s_1w_1=w_1s_2$ and the above expression is equal to 
\begin{align*} F\left(\Phi_1 \Phi_2 \mathbf{1}\right)&=\lim_{\substack{z=\left(\varepsilon,1+\varepsilon, -1,0\right)\\ \epsilon\to 0}} \left( \varphi_{2}^{21} \varphi_{1}^{23}\varphi_{2}^{13} \varphi_{2}^{31} \varphi_{3}^{41}\varphi_{2}^{43}\varphi_{0}^{2,-3}
\varphi_{1}^{4,-3}\varphi_{3}^{2,-1}\varphi_{0}^{4,-1}\varphi_{3}^{42}\right)\left(z\right)\cdot  \mathbf{1}\\
&=\lim_{\substack{z=\left(\varepsilon,1+\varepsilon, -1,0\right)\\ \epsilon\to 0}} \left( \varphi_{2}^{21} \varphi_{1}^{23} \cdot \frac{\left(2+\varepsilon\right)\varepsilon}{\left(1+\varepsilon\right)^2} \cdot  \varphi_{3}^{41}\varphi_{2}^{43}\varphi_{0}^{2,-3}
\varphi_{1}^{4,-3}\varphi_{3}^{2,-1}\varphi_{0}^{4,-1}\varphi_{3}^{42}\right)\left(z\right)\cdot  \mathbf{1}.
\end{align*}
The term $\varphi_{3}^{41}$ has a pole at $z=\chi_D$. It used to be in fusion with the term $\varphi_{2}^{31}$, but that term got cancelled. Rewrite the product as 
$$\lim_{\epsilon\to 0} \frac{\left(2+\varepsilon\right)}{\left(1+\varepsilon\right)^2} \left( \varphi_{2}^{21} \varphi_{1}^{23} \cdot  \left(\varepsilon \varphi_{3}^{41}\right)\cdot \varphi_{2}^{43}\varphi_{0}^{2,-3}
\varphi_{1}^{4,-3}\varphi_{3}^{2,-1}\varphi_{0}^{4,-1}\varphi_{3}^{42}\right)\left(z\right)\cdot  \mathbf{1},
$$
and notice that $\lim_{\epsilon\to 0} \varepsilon \varphi_{3}^{41}=\lim_{\epsilon\to 0}\varepsilon\left(s_3-\frac{1}{\varepsilon}\right)=-1$, while all other terms are regular at $\varepsilon\to 0$. So, to show that the above limit is $0$, it is enough to show that  
$$-2 \cdot \lim_{\epsilon\to 0} \left( \varphi_{2}^{21} \varphi_{1}^{23} \cdot   \varphi_{2}^{43}\varphi_{0}^{2,-3} \varphi_{1}^{4,-3}\varphi_{3}^{2,-1}\varphi_{0}^{4,-1}\varphi_{3}^{42}\right)\left(z\right)=0.
$$

The difficulty here is that after removing the factor $\varphi_{3}^{14}$, this is no longer a matched expression. To make it matched, let us replace all $\varphi_k^{\beta}$ that were in the product to the left of $\varphi_{3}^{14}$ by  $\varphi_k^{s_{14}\left(\beta\right)}$. This change corresponds to exchanging $z_1$ and $z_4$. As all terms are regular, and $\lim z_1=\lim z_4$, this does not change the limit. The resulting matched product is equal to 
$$-2 \cdot \lim_{\epsilon\to 0} \left( \varphi_{2}^{24} \varphi_{1}^{23} \varphi_{2}^{43}\varphi_{0}^{2,-3} \varphi_{1}^{4,-3}\varphi_{3}^{2,-1}\varphi_{0}^{4,-1}\varphi_{3}^{42}\right)\left(z\right)=-2 \cdot \lim_{\epsilon\to 0} \left( \varphi_{1}^{43} \varphi_{2}^{23} \varphi_{1}^{24}\varphi_{0}^{2,-3} \varphi_{1}^{4,-3}\varphi_{3}^{2,-1}\varphi_{0}^{4,-1}\varphi_{3}^{42}\right)\left(z\right).$$

Comparing this to the expression at the very beginning of this example, we see that the cancelling which we did has had the same effect as the following steps would have had: 
\begin{enumerate}
\item Delete all three terms of the fusion ($\varphi_{2}^{31} \varphi_{3}^{41}\varphi_{2}^{43}$);
\item Replace all $\varphi_k^{\beta}$ that show up before this fusion by $\varphi_k^{s_{pq}\beta}$, where $s_{pq}$ is the transposition permutation the deleted fusion would have achieved ($s_{14}$);
\item Multiply the limit by an overall function which has a regular, nonzero limit ($-2$).
\end{enumerate}

We have demonstrated above that the original limit is zero if and only if the limit of the modified product is zero. So, we are left with the task of showing that $$\lim_{\epsilon\to 0} \left( \varphi_{2}^{43} \varphi_{1}^{23} \varphi_{2}^{24}\varphi_{0}^{2,-3} \varphi_{1}^{4,-3}\varphi_{3}^{2,-1}\varphi_{0}^{4,-1}\varphi_{3}^{42}\right)(z)=0.$$ This is a matched product, with the $\alpha_{24}$ repeating root. Comparing $\overrightarrow{13}$ (which we have just cancelled) and $\overrightarrow{24}$ (which we are to cancel next), we see that the repeating root is again corresponding to two adjacent boxes in the same column of $D$, in the same two rows as before, and that these two boxes are one place to the right of the original ones. 

To cancel the two terms corresponding to $\alpha_{24}$, let $w_2=s_0s_1s_3s_0$, and notice that $s_1w_2=w_2s_3$, so the above limit is equal to 
$$\lim_{\epsilon\to 0} \left( \varphi_{2}^{43} \varphi_{1}^{23} \varphi_{0}^{4,-3} \varphi_{1}^{2,-3}\varphi_{3}^{4,-1}\varphi_{0}^{2,-1} \varphi_{3}^{24} \varphi_{3}^{42}\right)\left(z\right).$$
The product $\left(\varphi_{3}^{24} \varphi_{3}^{42}\right)\left(z\right)=\frac{\varepsilon\left(2+\varepsilon\right)}{\left(1+\varepsilon\right)^2}\to 0$ and all the other terms are regular, so this limit is $0$. This shows that $\Phi_1 \Phi_2 \mathbf{1}\in \Ker F$. 
\end{example}

These examples illustrate the general situation. The following lemma is the combinatorial heart of the proof, and the last part we need to prove the main theorem.

\begin{lemma}\label{column}
Assume that $D$ is a periodic skew diagram with no infinite column, $\kappa \ge 2$, and $F:M_D\to M_{\w^{-1}D}$ is a morphism of $\H$ modules defined in Proposition \ref{mapF}. For every $w\in \W$ such that $wT_0$ is standard, $s_{i}wT_0$ is not, and $i$ and $i+1$ are in the same column of $wT_0$, we have $$F\left(\Phi_i\Phi_w\mathbf{1}_D\right)=0.$$
\end{lemma}
\begin{proof}
{\bf Step 1.}
As in the proof of Lemma \ref{row}, we have
\begin{align*}
F\left(\Phi_i\Phi_w \mathbf{1}_D\right)&= \lim_{\substack{z\to \chi_D \\ z\in \mathcal{F}_D}}  \varphi_{s_iw\w}\left(\w^{-1}\left(z\right)\right)\mathbf{1}_{\w^{-1}\chi_D},
\end{align*}
with $l\left(s_iw\w\right)<\left(w\w\right)$. When writing $\varphi_{s_iw\w}\left(\w^{-1}\left(z\right)\right)=\varphi_i^{\beta }\left(z\right)\varphi_{w\w}\left(\w^{-1}\left(z\right)\right)$ as a product of $\varphi_{i_k}^{\beta_k}$, the root $\beta$ appears as $\beta=-\beta_k$ in one of the factors of $\varphi_{w\w}\left(\w^{-1}\left(z\right)\right)$. 

{\bf Step 2.} Specifying a point $z$ in the set $\mathcal{F}_D$ is equivalent to specifying a set of numbers $\varepsilon_a$, one for each row of $D$, and setting $z_i=\chi_i+\varepsilon_a$ if $i$ is in row $a$ of $T_0$. Taking the limit $\lim_{\substack{z\to \chi_D \\ z\in \mathcal{F}_D}}$ is then equivalent to taking $\lim_{\substack{\varepsilon_a\to 0 \\ \forall a}}$. This is well defined, as only finitely many $z_i$ and consequently finitely many $\varepsilon_a$ appear in the calculation of $\varphi_{\w}$.

 Let $i=wT_0\left(a,b\right)$, $i+1=wT_0\left(a+1,b\right)$, and $\varepsilon=\varepsilon_{a}-\varepsilon_{a+1}$. By Lemma \ref{regular}, the limit of $\varphi_{\w}\left(\w^{-1}\left(z\right)\right)$ at $z\to \chi_D$, $z\in \mathcal{F}_D$ exists, so it can be calculated as consecutive limits: 
 $$ \lim_{\substack{z\to \chi_D \\ z\in \mathcal{F}_D}} \varphi_{s_iw\w}\left(\w^{-1}\left(z\right)\right)=\lim_{\substack{\varepsilon_j\to 0 \\  \forall j}} \lim_{\varepsilon\to 0}\varphi_{s_iw\w}\left(\w^{-1}\left(z\right)\right).$$
We will prove that the inner limit $\lim_{\varepsilon\to 0}\varphi_{s_iw\w}\left(\w^{-1}\left(z\right)\right)$ is zero. 

{\bf Step 3.} Claim: if $\varphi_{k}^{-\alpha}\varphi_{w_1}\varphi_{j}^{\alpha}$ is matched, then it is equal to $\varphi_{w_1}\varphi_{j}^{-\alpha}\varphi_{j}^{\alpha}$. 

To prove this, let 
$\left(\varphi_{k}^{-\alpha}\varphi_{w_1}\varphi_{j}^{\alpha}\right)\left(z\right)=\varphi_{k}\left(w_1s_jw_2\left(z\right)\right)\varphi_{w_1}\left(s_jw_2\left(z\right)\right)\varphi_{j}\left(w_2z\right)$, and notice that $\varphi_{j}^{\alpha}={^{w_2}\varphi_{j}}$ implies $\alpha=w_2^{-1}\alpha_{j}$, and 
$\varphi_{k}^{-\alpha}={^{w_1s_jw_2}\varphi_{s_{k}}}$ implies $-\alpha=\left(w_1s_jw_2\right)^{-1}\alpha_{k}$. From this it follows that $\alpha_k=w_1\alpha_j$, and consequently that $w_1s_jw_1^{-1}=s_k$. So, $\varphi_{k}^{-\alpha}\varphi_{w_1}=\varphi_{s_kw_1}=\varphi_{w_1s_j}$, which implies the claim. 

If $\alpha=\alpha_{pq}$, and if $\varphi_{w_1}$ is written as a product of $\varphi_{j}^{\gamma}$, the change from $\varphi_{k}^{-\alpha}\varphi_{w_1}\varphi_{j}^{\alpha}$ to $\varphi_{w_1}\varphi_{j}^{-\alpha}\varphi_{j}^{\alpha}$ means that each factor of $\varphi_{w_1}$ changes from $\varphi^{\gamma}$ to $\varphi^{s_{pq}\left(\gamma\right)}$.

{\bf Step 4.} We will consider the following procedure on $\varphi_{s_iw\w}\left(\w^{-1}\left(z\right)\right)$, which we call \emph{cancellation at column $c$}.

Assume $c\ge b$ is such that $\left(a+1,c+1\right)\in D$. Then $\left(a,c\right), \left(a+1,c\right)\in D$. Let $i'=T_0\left(a,c\right)$, $j'=T_0\left(a+1,c\right)$, $j'+1=T_0\left(a+1,c+1\right)$. The product $\varphi_{s_iw\w}\left(\w^{-1}\left(z\right)\right)$, or any which was obtained from it by cancellation at columns $c'<c$, contains the fusion terms $\varphi_k^{j',i'}\varphi_{k+1}^{j'+1,i'}\varphi_k^{j'+1,j'}$. The procedure is: 
\begin{enumerate}
\item Delete the fusion terms $\varphi_k^{j',i'}$, $\varphi_{k+1}^{j'+1,i'}$, $\varphi_k^{j'+1,j'}$. 
\item For all terms $\varphi_{k'}^{\alpha}$ to the left of the deleted fusion, replace $\alpha$ by $s_{i',j'+1}\left(\alpha\right)$.
\item Multiply by the constant $-2$.
\end{enumerate}
  
{\bf Step 5.} Claim: if $\varphi_{w_1}\left(\w^{-1}z\right)$ was obtained from $\varphi_{s_1w\w}\left(\w^{-1}z\right)$ by cancellation at columns $b,b+1,\ldots c-1$, then it is matched, regular at $z=\chi_D$, and has the root $\alpha_{i'',j''}$, $i''=T_0\left(a,c\right)$, $j''=T_0\left(a+1,c\right)$, repeating.

To prove all these claims, we write $\varphi_{s_1w\w}\left(\w^{-1}z\right)$  as a product of $\varphi_k^{\beta}$ and look at what happens to the roots $\beta$ in cancellation at a column $c'$. Steps $(1)$ and $\left(2\right)$ ensure that the product is matched. Replacing some roots $\beta$ by $s_{i',j'+1}\left(\beta\right)$ does not change the roots at any fusions before column $c'$, nor does it create any new poles that would have to be resolved by fusion, so the new expression is regular. After cancellation at columns $b,b+1,\ldots c-1$, the root $\alpha_{i'',j''}$ appears twice in the new expression: 1)once in its original place, as given by the special order; this place was after the deleted fusions, so it not influenced by any cancellations; 2)the second time, in place where the root $\alpha_{i''-1',i''}$ was in the original expression $\varphi_{s_1w\w}\left(\w^{-1}z\right)$; this is before the last deleted fusion, and was changed by step $\left(2\right)$ of the cancellation at column $c$ from $\alpha_{i''-1',i''}$ to $s_{i''-1,j''}\left(\alpha_{i''-1',i''}\right)=\alpha_{i'',j''}$.

{\bf Step 6.} Assume that $\varphi_{s_1w\w}\left(\w^{-1}z\right)$ had been changed by cancellation at columns $b,b+1,\ldots c-1$. By Step 5, it has a repeating root $\alpha=\alpha_{i'',j''}$ for $i''=T_0\left(a,c\right)$, $j''=T_0\left(a+1,c\right)$. The new expression can be written as: 
 $$\varphi_{w_1s_kw_2s_jw_3}\left(\w^{-1}\left(z\right)\right)=\varphi_{w_1}\left(s_kw_2s_jw_3\w^{-1}\left(z\right)\right)\varphi_{k}^{\alpha}\left(z\right)\varphi_{w_2}\left(s_jw_3\w^{-1}\left(z\right)\right)\varphi_{j}^{-\alpha}\left(z\right)\varphi_{w_3}\left(\w^{-1}z\right).$$
 for some $w_1,w_2,w_3$. By Step 5, all terms in this product are regular. By Step 3, this can be further written as 
$$\varphi_{w_1}\left(s_kw_2s_jw_3\w^{-1}\left(z\right)\right)\varphi_{w_2}\left(w_3\w^{-1}\left(z\right)\right)\varphi_{j}^{\alpha}\left(z\right) \varphi_{j}^{-\alpha}\left(z\right)\varphi_{w_3}\left(\w^{-1}z\right).$$
We claim that $\varphi_{w_2}\left(w_3\w^{-1}\left(z\right)\right)$ is also regular at $z=\chi_D$. 

To prove that, look at the roots appearing in expressing $\varphi_{w_2}\left(w_3\w^{-1}\left(z\right)\right)$ as a product of $\varphi_k^{\beta}$. They are obtained from the roots appearing $\varphi_{w_2}\left(s_jw_3\w^{-1}\left(z\right)\right)$ by the action of $s_{i'',j''}$. This does not change any of the fusions in $\varphi_{w_2}\left(s_kw_3\w^{-1}\left(z\right)\right)$. It creates one new factor with a pole at $\varepsilon=0$, namely the factor with the root $\alpha_{i''-1,j''}=s_{i'',j''}\left(\alpha_{i''-1,i''}\right)$. In $\varphi_{w_2}\left(s_kw_3\w^{-1}\left(z\right)\right)$, this factor appears as: 
$$\ldots \varphi^{i''-1,i''}\varphi^{i''-1,i''+1}\varphi^{i''-1,i''+2} \ldots \varphi^{i''-1,k''}\varphi^{i''-1,j''-1}\ldots $$
where $k''$ is the last entry in row $a$. For $i''+1\le k' \le k''$, the term $\varphi^{k', j''-1}$ commutes with the first $k'-2$ terms of this product, and squares to $1-\frac{1}{z_{k'}-z_{j''-1}}\ne 0$. By adding the (normalized) squares of these elements to appropriate places, then commuting one copy to the beginning of the expression and using another to transform $\varphi^{k', j''-1} \varphi^{i''-1,k'}\varphi^{i''-1,j''-1}=\varphi^{i''-1,j''-1}\varphi^{i''-1,k'}\varphi^{k', j''-1}$, we can rewrite the above expression so that it contains 
$$\ldots \ldots \varphi^{i''-1,i''}\varphi^{i''-1,j''-1}\ldots $$
as adjacent terms. Once they are adjacent, by Lemma \ref{convex}, there is a reorder of the first part of it such that $\varphi^{j''-1,i''}\varphi^{i''-1,i''}\varphi^{i''-1,j''-1}$ are all adjacent. Then $\varphi^{i''-1,i''}\varphi^{i''-1,j''-1}$ can be fused, showing that there is no pole, and $\varphi_{w_2}\left(w_3\w^{-1}\left(z\right)\right)$ is regular at $\varepsilon=0$.

{\bf Step 7.} Assume that $\left(p=T_{0}\left(a,c\right), q=T_{0}\left(a+1,c\right)\right)$ is the last pair of boxes in rows $a$, $a+1$, i.e. $\left(a+1,c+1\right)\notin D$. Assume also that by doing cancelation in columns $b,\ldots c-1$ on $\varphi_{s_iw\w}\left(\w^{-1}\left(z\right)\right)$, we got a nonreduced expression where the root $\alpha_{pq}$ is repeated. (Here, we allow the possibility of $c=b$). This matched expression looks like
$$\varphi_{w_1s_kw_2s_jw_3}\left(\w^{-1}\left(z\right)\right)=\varphi_{w_1}\left(s_kw_2s_jw_3\w^{-1}z\right)\varphi_{k}^{pq}\left(z\right)\varphi_{w_2}\left(s_jw_3\w^{-1}z\right)\varphi_{j}^{qp}\left(z\right)\varphi_{w_3}\left(\w^{-1}z\right).$$
Because $\alpha_{pq}$ is the last vertical line between those two rows, $\varphi_{j}^{qp}$ is not the first term of any fusion, so by Step 5 $\varphi_{w_3}\left(\w^{-1}z\right)$ is regular. By Step 6, this can be written as
$$\varphi_{w_1}\left(s_kw_2s_jw_3\w^{-1}z\right)\varphi_{w_2}\left(w_3\w^{-1}z\right)\varphi_{j}^{pq}\left(z\right)\varphi_{j}^{qp}\left(z\right)\varphi_{w_3}\left(\w^{-1}z\right).$$
By Step 6, all terms in this product are regular at $\varepsilon=0$. The product $\varphi_{j}^{pq}\left(z\right)\varphi_{j}^{qp}\left(z\right)$ is equal to 
$$\varphi_{j}^{pq}\left(z\right)\varphi_{j}^{qp}\left(z\right)=\left(s_j+\frac{1}{z_p-z_q}\right)\left(s_j+\frac{1}{z_q-z_p}\right)=\left(s_j+\frac{1}{1+\varepsilon}\right)\left(s_j-\frac{1}{1+\varepsilon}\right)=\frac{\left(2+\varepsilon\right)\varepsilon}{\left(1+\varepsilon\right)^{2}},$$
which has the limit $0$ when $\varepsilon$ tends to $0$. Thus, 
$$\lim_{\substack{z\to \chi_D \\ z\in \mathcal{F}_D}} \varphi_{w_1s_kw_2s_jw_3}\left(\w^{-1}\left(z\right)\right)=0.$$

{\bf Step 8.} 
Assume that $\left(p=T_{0}\left(a,c\right), q=T_{0}\left(a+1,c\right)\right)$ is not the last pair of boxes in rows $a$, $a+1$, i.e. that $\left(a+1,c+1\right) \in D$. Assume also that after doing cancelation in columns $b,b+1,\ldots c-1$ on $\varphi_{s_iw\w}\left(\w^{-1}\left(z\right)\right)$, we got a nonreduced expression where the root $\alpha_{pq}$ is repeated. The second appearance of this root is the first term of a fusion. This matched expression can be written as:
\begin{align*}
& \varphi_{w_1s_kw_2s_js_{j+1}s_jw_3}\left(\w^{-1}\left(z\right)\right)  = \quad &  &  \\
& \quad =\left( \left(^{s_kw_2s_js_{j+1}s_jw_3\w^{-1}}\varphi_{w_1}\right)\cdot \varphi_{k}^{{pq}} \cdot \left(^{s_js_{j+1}s_jw_3\w^{-1}}\varphi_{w_2}\right) \cdot \varphi_{j}^{qp} \varphi_{j+1}^{q+1,p}\varphi_j^{q+1,q} \cdot\left(^{\w^{-1}}\varphi_{w_3}\right)\right)\left(z\right) && \\
& \quad =\left( \left(^{s_kw_2s_js_{j+1}s_jw_3\w^{-1}}\varphi_{w_1}\right)\cdot \left(^{s_{j+1}s_jw_3\w^{-1}}\varphi_{w_2}\right) \cdot \varphi_{j}^{{pq}}  \varphi_{j}^{qp} \varphi_{j+1}^{q+1,p}\varphi_j^{q+1,q} \cdot\left(^{\w^{-1}}\varphi_{w_3}\right)\right)\left(z\right).
\end{align*}
By Step 6, the terms $\varphi_{w_1}$, $\varphi_{w_2}$, $\varphi_{w_3}$ of this product are regular. Let us calculate the limit of the middle term: 
\begin{align*}
&\lim_{z\to \chi_D\in \mathcal{F}_D} \left(\varphi_{j}^{{pq}}  \varphi_{j}^{qp} \varphi_{j+1}^{q+1,p}\varphi_j^{q+1,q} \right)\left(z\right)= \quad && \\
& \qquad =\lim_{\varepsilon\to 0} 
\left(s_j+\frac{1}{z_p-z_q}\right)\left(s_j+\frac{1}{z_q-z_p}\right)\left(s_{j+1}+\frac{1}{z_{q+1}-z_{p}}\right)\left(s_j+\frac{1}{z_{q+1}-z_{q}}\right) &&\\
& \qquad =\lim_{\varepsilon\to 0} \left(s_j+\frac{1}{1+\varepsilon}\right)\left(s_j-\frac{1}{1+\varepsilon}\right)\left(s_{j+1}-\frac{1}{\varepsilon}\right)\left(s_j+1\right) &&\\
& \qquad =\lim_{\varepsilon\to 0} \frac{\left(2+\varepsilon\right)\cdot \varepsilon}{\left(1+\varepsilon\right)^2}\left(s_{j+1}-\frac{1}{\varepsilon}\right)\left(s_j+1\right)=-2\left(s_j+1\right)=-2 \varphi_{j}^{\alpha_{q+1,q}}. &&
\end{align*}

Substituting this in the original expression, we get  
\begin{align*}
& \lim_{z\to \chi_D\in \mathcal{F}_D} \varphi_{w_1s_kw_2s_js_{j+1}s_jw_3}\left(\w^{-1}\left(z\right)\right)= && \\
& \qquad =\lim_{z\to \chi_D\in \mathcal{F}_D} \left( \left(^{s_kw_2s_js_{j+1}s_jw_3\w^{-1}}\varphi_{w_1}\right)\cdot \left(^{s_{j+1}s_jw_3\w^{-1}}\varphi_{w_2}\right) \cdot \left(-2\varphi_j^{q+1,q}\right) \cdot\left(^{\w^{-1}} \varphi_{w_3}\right)\right)\left(z\right).&&
\end{align*}
This product is not matched (as we effectively took $\varphi_{{j+1}}^{q+1,p}$ out). We use the fact that $
\lim z_p=\lim z_{q+1}$ and that all terms are regular to replace the above limit of a non-mathched product with a limit of the following matched product:
\begin{align*}
& \lim_{z\to \chi_D\in \mathcal{F}_D} \varphi_{w_1s_kw_2s_js_{j+1}s_jw_3}\left(\w^{-1}\left(z\right)\right)= && \\
& \qquad =\lim_{z\to \chi_D\in \mathcal{F}_D} \left( \left(^{s_{j+1} s_kw_2s_js_{j+1}s_jw_3\w^{-1}}\varphi_{w_1}\right)\cdot \left(^{s_jw_3\w^{-1}}\varphi_{w_2}\right) \cdot \left(-2\varphi_j^{q+1,q}\right) \cdot\left(^{\w^{-1}}\varphi_{w_3}\right)\right)\left(z\right) && \\
& \qquad =-2 \lim_{z\to \chi_D\in \mathcal{F}_D} \left( \left(^{s_{j+1} s_kw_2s_js_{j+1}s_jw_3\w^{-1}}\varphi_{w_1}\right)\cdot \varphi_{k}^{q+1,q} \cdot  \left(^{s_k s_jw_3\w^{-1}}\varphi_{w_2}\right) \cdot\left(^{\w^{-1}}\varphi_{w_3}\right)\right)\left(z\right) && \\
& \qquad =-2 \lim_{z\to \chi_D\in \mathcal{F}_D} \varphi_{w_1s_kw_2 w_3}  \left(\w^{-1}z\right). &&
\end{align*}

Comparing the beginning and the end of this computation, we see that the net effect of Step 8 was the cancelation at column $c$ as described in Step 4. 

{\bf Step 9.} Finally, let us put it all together. For $D$ as in the statement, and $i$, $i+1$ in the same column of $wT_0$, we calculate $F\left(\Phi_i\Phi_w\mathbf{1}_D\right)$ by showing $\lim_{\varepsilon\to 0} \varphi_{s_iw\w}\left(\w^{-1}z\right)=0$. Let $i=wT_{0}\left(a,b\right)$ $i+1=wT_0\left(a+1,b\right)$, $p=T_{0}\left(a,b\right)=w^{-1}(i)$ and $q=T_0\left(a+1,b\right)=w^{-1}\left(a+1,c\right)$. Then the root $\alpha_{pq}$ is repeated in the expression of $\varphi_{s_iw\w}$; it appears once as an exponent of $\varphi_{s_i}$ and once among the exponents of $\w$. If $\left(a,b\right)$ and $\left(a+1,b\right)$ are not the last pair of boxes in that row, then by Step 9 this limit is equal to the limit of a similar matched product, obtained from $\varphi_{s_iw\w}\left(\w^{-1}z\right)$ by cancelation at column $b$, which has a repeating root $\alpha_{p+1,q+1}$, corresponding to $p+1=T_0\left(a,b+1\right)$ and $q+1=T_0\left(a+1,b+1\right)$. Repeating this several times if necessary, we get that $\lim_{\varepsilon\to 0} \varphi_{s_iw\w}\left(\w^{-1}z\right)$ is  equal to the limit of some other matched product, obtained from $\varphi_{s_iw\w}\left(\w^{-1}z\right)$ by cancelations at all columns $b\le c' <c$, where $\left(a,c\right)$ and $\left(a+1,c\right)$ are the last pair of boxes these two rows, and the root $\alpha_{p+\left(c-b\right), q+\left(c-b\right)}$, $p+\left(c-b\right)=T_0\left(a,c\right)$, $q+\left(c-b\right)=T_0\left(a+1,c\right)$, is repeated. By Step 8, this limit is equal to $0$.
\end{proof}

\begin{proof}[Proof of Proposition \ref{two-way}]
For $\w$ constructed in Lemma \ref{defw}, in Proposition \ref{mapF} we constructed homomorphism $F:M_D\to M_{\w^{-1}\chi_D}$ of Verma modules. To see that it factors through the kernel of the surjective map $Q:M_D\to L_D$, it is, by Theorem \ref{defss}, enough to show that $F\left(\Phi_i \Phi_w \mathbf{1}_D\right)=0$ for all $w,s_i$ such that $wT_0$ is standard and $s_iwT_0 $is not standard. By Lemma \ref{adjacent}, in that case $i$ and $i+1$ are adjacent in $wT_0$, either in the same row or in the same column. If they are in the same row, then $F\left(\Phi_i \Phi_w \mathbf{1}_D\right)=0$ by Lemma \ref{row}. If they are in the same column, then $F\left(\Phi_i \Phi_w \mathbf{1}_D\right)=0$ by Lemma \ref{column}. So, $F$ factors through $Q$, and induces a nonzero homomorphism $L_D\hookrightarrow M_{\w^{-1}\chi_D}$.
\end{proof}

\section{The case $\kappa=1$}\label{kappa1}

In this section, we show that for $n\ge 2$ and $\kappa=1$, every semisimple irreducible module $L_D$ for $\H$ can be embedded into a Verma module. The choice of this Verma module and consequently the method of proof is different than in Section \ref{yescando}; in particular, the existence and properties of the element $\w\in \W$ do not carry over to $\kappa=1$ case. The construction in this section is different and straightforward. The first step is to show that the for $\kappa=1$ and fixed $n$, there are very few periodic skew diagrams.

\begin{lemma}
Let $n\ge 2$, $D$ a periodic skew diagram of degree $n$ and period $\left(m,-l\right)$, and $\kappa=m+l=1$. Then the fundamental domain of $D$ is one row with $n$ consecutive boxes, while  $D$ consists of $n$ consecutive infinite columns.
\end{lemma}
\begin{proof}
If $D$ is a periodic skew diagram, then $l\ge 0$, which together with $m\ge 1$ and $m+l=1$ implies $m=1$, $l=0$.
\end{proof}

In Section \ref{nocando} we showed that the same picture with more than one row in the fundamental domain of $D$ produces irreducible modules with torsion which cannot be embedded into a Verma module. The proof used elements associated to different rows of the fundamental domain of $D$ to find torsion. Here, we show that for if the fundamental domain has only one row, this is not the case, and we find an explicit embedding of $L_D$.

\begin{lemma}\label{Katkappa1}
Let $m=1$, $l=0$, $\kappa=1$, $\mu \in \ZZ$, $\lambda=\mu+n-1$, and let $D$ be the periodic skew diagram consisting of $n$ consecutive infinite columns, $D=\ZZ \times [\mu,\lambda]$. Then the irreducible $\ddot{H}(1)$ module $L_D$ is equal to the small Verma module $N_D$.
\end{lemma}
\begin{proof}
Let us first describe $N_D$. The fundamental domain of $T_0$ is $\raisebox{-2pt}{\young(123\hfill \cdots \hfill n)}$, so $I=\{1,\ldots, n-1\}$. The small Verma module is the quotient $N_D=M_D/K$, where $K$ is the $\H$ submodule generated by $\Phi_i \One_D$ for $i=1,\ldots,n-1$.

Let $Q$ be the quotient map $Q:M_D\to L_D$. It factors through the surjection $M_D\to N_D$, so $K\subseteq \Ker Q$. We claim that $K=\Ker Q$.

As left $\CC[\W]$ modules, $M_D\cong \CC[\W]$, so any element of $M_D$ can be uniquely written as a finite sum $v=\sum_{w\in \W} a_w w\One_D$ for some $a_w \in \CC$. Let $k\left(v\right)$ be the length of the longest $w\in \W$ with nonzero $a_w$, and let $k'\left(v\right)$ be the number of terms $a_w w$ with nonzero $a_w$ and $l\left(w\right)=k$.

Assume that $K$ is a proper subset of $\Ker Q$. Consider the subset of $\Ker Q \setminus K$ consisting of elements which have minimal $k\left(v\right)=k$, and among such elements, pick one with the minimal $k'\left(v\right)$. Write it as $v=\sum_{w\in \W} a_w w\One_D$.

By the comments below Theorem \ref{defss}, for every $w\One_D \in M_D$ there exist $b_{w'}, c_{w'}\in \mathbb{C}$ such that 
$$Q\left(w\One_D\right)=v_{wT_0}+\sum_{l\left(w'\right)<l\left(w\right)} b_{w'} v_{w'T_0}, \qquad \textrm { if }w T_0 \textrm{ is standard },$$
$$Q\left(w\One_D\right)=\sum_{l\left(w'\right)<l\left(w\right)} c_{w'} v_{w'T_0}, \qquad \textrm { if }w T_0 \textrm{ is not standard }.$$
As $v\in \Ker Q$, we see that 
\begin{align*}
0&=Q\left(v\right)\\
&=Q\left( \sum_{l\left(w\right)=k} a_w w+ \sum_{l\left(w\right)<k} a_w w\right)\\
&= \sum_{\substack{l\left(w\right)=k \\ w T_0 \textrm{ standard}}} a_w v_{wT_0}+\sum_{l\left(w\right)<k} d_w v_{wT_0}.
\end{align*}

As the set $\{v_{wT_0}| wT_0 \textrm{standard}\}$ is a basis of $L_D$, it follows that $wT_0$ is not standard for all nonzero leading terms $a_w w$ of $v$ with $a_w\ne 0$, $l\left(w\right)=k$. As $k\left(v\right)=k$, the set of such terms is nonempty; in fact it has $k'\left(v\right)$ elements. 

Choose such a summand $a_ww$. The fundamental domain of $D$ has only one row, so all periodic tableaux are column increasing. Therefore, $wT_0$ is not row increasing. The fundamental domain of $wT_0$ is $${\Large \begin{ytableau} \scriptstyle  w(1) & \scriptstyle  w \left(2\right) & \none[\dots] & \scriptstyle  w\left(n\right) \end{ytableau}}$$ so there exists $i\in 1,2,\ldots n-1$ such that $w (i) > w \left(i+1\right)$. From this it follows that $w\alpha_i$ is a negative root, so by \cite{H} Lemma ~1.6,  $l\left(ws_i\right)<l\left(w\right)$ and by the Exchange Condition ~1.7. from \cite{H} there exists a reduced expression of $w$ ending in $s_i$, $w=w's_i$, $l\left(w'\right)=k-1$.

The element $w'\Phi_i\One_D=w'\left(s_i-1\right)\One_D$ is in $K\subseteq \Ker Q$. Consider $$v-a_w\left(w'\left(s_i-1\right)\right)\One_D=v-a_ww\One_D+a_ww'\One_D.$$ It is in $\Ker Q$, as both $v$ and $w'\Phi_i\One_D$ are. It is not in $K$, because $w'\Phi_i\One_D$ is and $v$ is not. Finally, it has $k'-1$ summands of length $k$. This contradicts the choice of $v$ as minimal with said properties and the assumption that $K\ne \Ker Q$.
\end{proof}

\begin{proposition}Let $D=\ZZ\otimes [\mu,\mu+n-1]$ be a periodic skew Young diagram of period $\left(1,0\right)$, and let $w_0$ be the longest element of the symmetric group $W_n$. Let $\chi=\chi_D$ and $\tau=w_0 \chi$. Then the map $M_D\to M_{\tau}$ of $\ddot{H}_n(1)$ modules determined by $\One_D \mapsto \Phi_{w_0}\One_\tau$ factors through $L_D$, and gives rise to an inclusion $L_D\hookrightarrow M_\tau$. \label{thirdcase}
\end{proposition}
\begin{proof}
The fundamental domain of $T_0$ is $\raisebox{-2pt}{\young(123\hfill \cdots \hfill n)}$, so 
\begin{align*}
\chi&=\left(0,1,2,\ldots, n-1\right)+ \left(\mu-1\right)\left(1,1,\ldots ,1\right)\\
\tau&= \left(n-1,n-2,\ldots, 1,0\right)+ \left(\mu-1\right)\left(1,1,\ldots ,1\right).
\end{align*}

The vector $\Phi_{w_0}\One_\tau \in M_\tau$ is well defined (its factors have no poles because all the coordinates of $\tau$ are distinct), nonzero (as it is of the form $\Phi_{w_0}\One_\tau=w_0\One_\tau +\sum_{l\left(w\right)<l\left(w_0\right)}a_{w}w\One_{\tau}$), and an eigenvector for $\left(u_1,\ldots , u_n\right)$ with the eigenvalue  $w_0\tau=w_0^2\chi=\chi$. So, there is a unique nonzero morphism of $\H$ modules  $F: M_D\to M_\tau$ determined by $F\left(\One_D\right)=\Phi_{w_0}\One_\tau$.

We claim that $F$ is zero when restricted to $\Ker Q$, for $Q:M_D\to L_D$ the quotient map. By the previous lemma, it is enough to see that $F\left(\Phi_i\One_D\right)=0$ for $i=1,\ldots, n-1$. Pick a reduced reduced decomposition of the longest element $w_0$ starting with $s_i$, $w_0=s_iw$, $l\left(w\right)=l\left(w_0\right)-1$. Then $\Phi_{w_0}=\Phi_i \Phi_{w}$. As $\Phi_{w}\One_\tau$ is an eigenvector 
with eigenvalue $w\tau=s_i\chi$, and $\left(s_i\chi\right)_i-\left(s_i\chi\right)_{i+1}=\chi_{i+1}-\chi_i=1$, we have 
$$F\left(\Phi_i\One_D\right)=\Phi_i\Phi_{w_0}\One_\tau= \Phi_i^2\Phi_{w}\One_\tau=\left(s_i-1\right)\left(s_i+1\right)\Phi_{w}\One_\tau=0.$$

So, the map $F:M_D\to M_\tau$ is zero on the kernel of the quotient map $Q:M_D\to L_D$, and so it induces a nonzero $\H$ morphism $L_D\to M_\tau$. As $L_D$ is irreducible, this map is an inclusion. 
\end{proof}

%%%%%%%%%%%%%%%%%%%%%%%%


\begin{thebibliography}{EV1}
\begin{normalsize}

\bibitem[C]{C}I. Cherednik, \emph{Special bases of irreducible finite-dimensional representations of the degenerate affine Hecke algebra},  Funktsional. Anal. i Prilozhen. 20 (1986), no. 1, 87--88.

\bibitem[D]{D}M. Dyer, \emph{Hecke algebras and shellings of Bruhat intervals}, Compositio Math. 89 (1993), no. 1, 91--115.

\bibitem[GNP]{GNP} V. Guizzi, M. Nazarov and P. Papi, \emph{Cyclic generators for irreducible representations of affine Hecke algebras}, J. Combin. Theory Ser. A 117 (2010), no. 6, 683--703. 

\bibitem[GP]{GP} V. Guizzi and  P. Papi, \emph{A  combinatorial  approach to the fusion process for the symmetric group}, European J. Combin. 19 (1998), no. 7, 835--845.

\bibitem[H]{H} J. Humphreys, \emph{Reflection groups and Coxeter groups}, Cambridge Studies in Advanced Mathematics, 29. Cambridge University Press, Cambridge, 1990.

\bibitem[L]{L} G. Lusztig, \emph{Some examples of square integrable representations of semisimple p-adic groups}, Trans. Amer. Math. Soc. 277 (1983), no 2, 623--653.

\bibitem[KNP]{KNP} S. Khoroshkin, M. Nazarov and P. Papi, \emph{Irreducible representations of Yangians}, J. Algebra 346 (2011), 189--226.

\bibitem[KS]{KS}F. Knop and S. Sahi, \emph{A recursion and a combinatorial formula for Jack polynomials}, Invent. Math. 128 (1997), no. 1, 9--22.

\bibitem[M]{M} I. G. Macdonald, \emph{Affine Hecke algebras and orthogonal polynomials}, Cambridge Tracts in Mathematics, 157. Cambridge University Press, Cambridge, 2003.

\bibitem[P1]{P1}P. Papi, \emph{A characterization of a special order in a root system}, Proc. Amer. Math. Soc. 120 (1994), no. 3, 661--665. 

\bibitem[P2]{P2}P. Papi, \emph{Inversion tables and minimal left coset representatives for Weyl groups of classical type}, 
J. Pure Appl. Algebra 161 (2001), no. 1-2, 219--234.

\bibitem[R]{R} A. Ram, \emph{Affine Hecke algebras and generalized standard Young tableaux},  J. Algebra 260 (2003), no. 1, 367--415.

\bibitem[Ro]{Ro} J. Rogawski, \emph{On modules over the Hecke algebras of a p-adic group}, Invent. Math. 79 (1985), 443--465.

\bibitem[S1]{S1} T. Suzuki, \emph{Rational and trigonometric degeneration of the double affine Hecke algebra of type $A$}, Int. Math. Res. Not. 2005, no. 37, 2249--2262. 

\bibitem[S2]{S2} T. Suzuki, \emph{Classification of simple modules over degenerate double affine Hecke algebras of type $A$}, Int. Math. Res. Not. 2003, no. 43, 2313--2339. 

\bibitem[SV]{SV} T. Suzuki and M. Vazirani, \emph{Tableaux on periodic skew diagrams and irreducible representations of the double affine Hecke algebra of type $A$}, Int. Math. Res. Not. 2005, no. 27, 1621--1656. 

\bibitem[Z]{Z} A. Zelevinsky, \emph{Induced representations of reductive p-adic groups II. On irreducible representations of $GL(n)$}, Ann. Sci. \'Ecole Norm. Sup. (4) 13 (1980), no. 2, 165--210. 




\end{normalsize}
\end{thebibliography}
\end{document}